\documentclass[]{amsart}
\usepackage{amsmath}
\usepackage{amsfonts}
\usepackage{amssymb}
\usepackage{amsthm}
\usepackage{enumerate}
\usepackage{enumitem}
\usepackage{graphicx,color,xcolor,tikz}
\usepackage{bbold}
\usepackage[utf8]{inputenc}
\usepackage{hyperref}
\hypersetup{
    colorlinks=true,                         
    linkcolor=blue, 
    citecolor=red, 
  } 

\newtheorem{Lem}{Lemma}

\newtheorem{Prop}{Proposition}

\newtheorem{Rem}{Remark}

\newcommand{\p}{\partial}

\renewcommand{\dfrac}[2]{\frac{#1}{#2}}
\newcommand{\demi}{\frac{1}{2}}

\title{Asymptotic preserving schemes for Hyperbolic Systems with Relaxation}
\date{\today}

\author{C. Mahmoud, H. Mathis}
\address{Institut Montpelliérain Alexander Grothendieck, ANGUS, Université de
  Montpellier, Inria, CNRS, Montpellier, France}
\email{christina.mahmoud@umontpellier.fr , helene.mathis@umontpellier.fr}

\begin{document}

\begin{abstract}
  This paper presents the construction of two numerical schemes for the
  solution of hyperbolic systems with relaxation source terms. The
  methods are built by considering
  the relaxation system as a whole, without separating the resolution
  of the convective part from that of the source term.
  The first scheme combines the centered FORCE approach of Toro and
  co-authors with the
  unsplit strategy proposed by Béreux and Sainsaulieu.
  The second scheme consists of an approximate Riemann solver which
  carefully handles the source term approximation.
  The two
  schemes are built to be asymptotic preserving, in the sense that
  their limit schemes are consistent with the equilibrium model as the relaxation
  parameter tends to zero, without any CFL restriction. For specific
  models, it is possible to prove that they preserve invariant domains
  and admit a discrete entropy inequality.
\end{abstract}

\maketitle

\noindent
\textbf{Key-words.}  Hyperbolic systems with relaxation, asymptotic
preserving schemes, approximate Riemann solver

\noindent
\textbf{2020 MCS.} 35L40, 65M08, 65M12

\tableofcontents

\section{Introduction}

We are interested in the numerical approximation of
hyperbolic systems with relaxation. Such systems are a class of partial differential
equations that model multiscale phenomena and where nonlinear hyperbolic
convection interacts with relaxation mechanisms.
These mechanisms are modeled by nonlinear relaxation terms
involving relaxation parameter, denoted $\varepsilon$ in the sequel.
As this scaling parameter tends to zero, solutions are driven towards equilibrium solutions.
Contrary to dynamical systems, if the initial Cauchy data of the
relaxation model belongs to the equilibrium manifold, then the
solution could be out of equilibrium \cite{dafermosBook}.
In order to study the stability and the convergence of solutions to
hyperbolic systems with relaxation towards their equilibrium
hyperbolic models, several criteria have been established
\cite{Liu87, chen94:relax, Bouchut05}. The strongest criterion relies on the Lax entropy
structure of the hierarchy of models: if the relaxed system is endowed
with an entropy-flux pair, which dissipates the source term, then the
restriction of the pair to the
equilibrium manifold is an entropy-flux pair for the equilibrium
model. The models we will consider for applications fall into this category.

The literature on numerical schemes for hyperbolic systems with
relaxation is extensive.
This is largely due to the fact that relaxation techniques were
originally introduced as a means to develop robust schemes for
homogeneous hyperbolic systems. A seminal contribution in this area is
the work of Jin and Xin \cite{jin1995relaxation}, who proposed a relaxation-based approximation
for systems of conservation laws.
A fundamental and robust strategy \cite{JinAP}, known as the splitting method,
involves decoupling the convective part—handled via a numerical flux
(such as HLL)—from the source term, which is treated implicitly to
ensure the correct asymptotic behavior as $\varepsilon \to 0$.
In the case of the Jin–Xin model, the resulting scheme is uniformly
convergent with respect to both the relaxation parameter $\varepsilon$
and the discretization parameters \cite{LS01,FR13}.

Several other methods build on this foundational approach and have
shown excellent performance for kinetic models. In particular, IMEX
(Implicit–Explicit) schemes have become widely recognized for their
effectiveness in handling stiff source terms.
These methods treat the non-stiff hyperbolic fluxes explicitly and the stiff relaxation source terms implicitly, enabling stable time integration without resolving the fast time scales.
IMEX schemes are designed to be \textit{asymptotic-preserving} (AP),
meaning they remain stable and consistent as $\varepsilon \to 0$, and
in many cases, they  preserve the correct order of accuracy in the
limiting regime.
These properties have been rigorously analyzed in several works,
including the unified framework presented by Boscarino, Pareschi, and
Russo~\cite{boscarino2017imex}, the uniform stability and accuracy
results for linear systems by Hu and Shu~\cite{hu2023uniform} and Ma
and Huang~\cite{ma2024uniform}, as well as the comprehensive review of
AP methods for quasilinear hyperbolic systems provided by Boscarino
and Russo~\cite{boscarino2024ap}.

The numerical schemes we propose here differ from the original IMEX
approach, even if they rely on explicit/implicit treatments.
In fact, they can be viewed as two extensions of the original
scheme proposed
in
\cite{bereux1996zero} and \cite{bereux1997roe}.
Those papers present a staggered-grid scheme 
with three steps: first, a shift followed by solving the source-term
ODEs; second, the implicit computation of the average
solution over a half time step; and finally, the repetition of the
two first steps
to estimate the solution at time $n+1$.
In the original works, the numerical fluxes are based on Roe-type
approximation.
The method differs from standard operator-splitting strategies because
it is built directly on the fully coupled relaxation system. In
particular, an approximate solution of the source term, computed in
the first step, is used as an input to the numerical scheme.

Here are designed two possible first-order adaptations of this method.

First we propose replacing the Roe-type approximation with the \textit{First-Order Centered} FORCE scheme,
one of whose earliest references is \cite{toro1992force}.
FORCE is defined as the average of the Lax–Friedrichs and Richtmyer
schemes, aiming to combine the stability of the former
with the improved resolution of the latter. Unlike classical upwind
methods, FORCE avoids solving Riemann problems while still preserving
the conservative structure of the equations.We focus here on the
first-order method, but it has later been extended to second-order
accuracy within the MUSCL framework via nonlinear slope
limiters, thereby enforcing the total variation diminishing
   property \cite{toro2009riemann}. In addition, Chen and Toro
\cite{chen2003centred} proved that the FORCE scheme satisfies a
\textit{fully discrete entropy inequality}, ensuring convergence
toward the physically admissible entropy solution.
The resulting method has desirable properties: it is
asymptotic preserving, and at equilibrium, the limit scheme corresponds
to a FORCE scheme applied to the equilibrium model.

Second, we incorporate the solution of the source term into the
definition of an approximate Riemann solver, in
the spirit of \cite{berthon2016well}. In that work, the authors proposed an
explicit approximate Riemann solver initially designed to preserve the
stationary states of a convection-diffusion model. The scheme is
based on the integral consistency relation with the solution of the
Riemann problem. For a conservative equation, it is possible to
determine the exact average solution. However, in the presence of a
source term, this calculation becomes complex, and the technique
proposed in \cite{berthon2016well} allows to take into account the
influence of the source term in the definition of the Riemann solver.
Our second numerical scheme combines  this technique for defining the Riemann solver with the implicit
computation of the source term proposed in \cite{bereux1996zero} and \cite{bereux1997roe}.
The overall method is asymptotic preserving by construction. For
the Jin and Xin model, it can be shown to be entropy-satisfying and to preserve invariant domains.

The paper is organized as follows.
Section \ref{sec:continuous-setting} presents the main properties of
hyperbolic systems with relaxation and some exemples on which the
numerical schemes will be compared, namely the Jin-Xin model, the
Chaplygin model and an homogeneous two-phase model.
In Section \ref{sec:staggered-scheme} we construct the staggered
scheme that combines the centered approximation techniques of Toro and
coauthors with the approach of Béreux and Sainsaulieu, in which  numerical
fluxes are evaluated on states obtained by the resolution of the
source term. The scheme inherits the properties of the FORCE scheme,
namely consistency and a discrete entropy inequality. It is also
proved to be asymptotic-preserving and to preserve the invariant
domain for the equilibrium Jin–Xin model. The definition of the
approximate Riemann solver is addressed in Section
\ref{sec:ARS}. Following the Harten–Lax–van Leer methodology, we
impose integral consistency constraints to guarantee both consistency
and a discrete entropy inequality. To ensure the desired asymptotic
behavior, a correction is applied at the Godunov projection step.
Preservation of invariant domains and a local entropy inequality are
proven in the case of the Jin and Xin model. Finally, Section
\ref{sec:numerical-results} presents numerical tests that illustrate
the asymptotic-preserving properties of the two schemes.


\section{Continuous setting}
\label{sec:continuous-setting}

In this section, we summarize the main features of hyperbolic systems with relaxation. For clarity, we restrict attention to the one-dimensional setting, which streamlines the presentation of the numerical schemes introduced in the next section. For a general multidimensional framework, we refer the reader to
\cite{Liu87, chen94:relax, HN03, Yongconvergence, tzavaras}.

We also present three examples of systems on which numerical simulations
will be carried out in Section \ref{sec:numerical-results}.

\subsection{The general case}
\label{sec:general-case}
  
We are interested in hyperbolic systems with relaxation of the form
\begin{equation}
    \label{eq:HSR}
    \p_t \mathbf{W} + \p_x \mathbf{f}(\mathbf{W}) =
    \dfrac{1}{\varepsilon}\mathbf{R}(\mathbf{W}).
  \end{equation}
The vector of conservative variables $\mathbf{W}: \mathbb R^+ \times
\mathbb R$ takes values in a convex set of admissible
states $K \subset \mathbb R^n$. The flux function $\mathbf{f}$ is such
that, for each $\mathbf{W}\in K$, the Jacobian matrix 
$\nabla \mathbf f(\mathbf{W})$ has read eigenvalues $\lambda_i$, $i=1,\ldots,n$
\begin{equation*}
\lambda_1 \leq \lambda_2 \leq \dots \leq \lambda_n,
\end{equation*}
and is diagonalizable over $\mathbb{R}$ with  a complete
set of $n$
linearly independent eigenvectors.

The source term and the relaxation time $\varepsilon$ govern the
behavior of the system's solutions.
The stability of solutions and their behavior as 
$\varepsilon$ tends to zero have been the subject of numerous studies. 
Following \cite{Liu87, chen94:relax, HN03,
  Yongconvergence, tzavaras}, we assume there exists
a linear operator $M_1 : \mathbb R^n \to \mathbb R^k$  of rank $k\leq n$ such that
\begin{equation}
      \label{eq:M1}  
      M_1 \mathbf R(\mathbf W) = 0, \quad \forall \mathbf W \in K.
\end{equation}
The operator $M_1$ defines the conserved variables $\mathbf
W^{(1)}=M_1 \mathbf W$, that satisfy
\begin{equation}
    \label{eq:HSR1}
    \p_t \mathbf{W}^{(1)} + \p_x M_1\mathbf{f}(\mathbf{W}) = 0.
\end{equation}
There also exists a  linear operator $M_2 : \mathbb R^n \to \mathbb R^{n-k}$  of
rank $n-k$ exists such that
the operator $M=\begin{pmatrix}M_1\\M_2 \end{pmatrix}$ is nonsingular.
Setting $\mathbf{W}^{(2)}=M_2 \mathbf W$ and defining
\begin{equation}
  \label{eq:fk}
  \mathbf f^{(k)}(\mathbf W):=M_k\mathbf f(\mathbf W), \quad \mathbf
  R^{(k)}(\mathbf W)=M_k\mathbf R(\mathbf W),
  \quad k=1,2,
\end{equation}
the system \eqref{eq:HSR} can be rewritten as
\begin{equation}
  \label{eq:HSR2}
  \begin{cases}
    \p_t \mathbf{W}^{(1)} + \p_x \mathbf{f}^{(1)}(\mathbf{W}) =0,\\
    \p_t \mathbf{W}^{(2)} + \p_x \mathbf{f}^{(2)}(\mathbf{W}) =
    \dfrac{1}{\varepsilon}\mathbf{R}^{(2)}(\mathbf{W}).
  \end{cases}
\end{equation}
We may also use the notation $\mathbf f^{(k)}(\mathbf W^ {(1)},
\mathbf W^ {(2)})=\mathbf f^{(k)}(\mathbf W)$ in order to highlight the
dependence of the flux.
In the following we will consider that
\begin{equation}\label{condition jacob}
\nabla \mathbf{f}^{(2)}(\mathbf{W})\mathbf{R}(\mathbf{W})=0 
\end{equation}
and we focus on a specific expression of source terms, namely linear source terms in $\mathbf{W}^{(2)}$
\begin{equation}
\label{eq:ST-0}
\mathbf{R}^{(2)}(\mathbf{W}) = \mathbf Q(\mathbf{W}^{(1)}) -
\mathbf{W}^{(2)},
\end{equation}
where $\mathbf Q : \mathbb R^k \to \mathbb R^{n-k}$ may be nonlinear.

We assume there exists an \emph{equilibrium map}
\(E:M_1K\to K\) whose image is the equilibrium manifold associated with \eqref{eq:HSR}, namely
\begin{equation}
\mathcal M_{eq}:=\{\mathbf W\in K:\ \mathbf R(\mathbf W)=0\}.
\end{equation}
In particular, \(\mathcal M_{eq}\) can be parameterized by the conserved variables \(\mathbf W^{(1)}\in M_1K\). For source terms of type \eqref{eq:ST-0}, the equilibrium manifold is
simply  given by $\mathbf W \in K $ such that
\begin{equation}
  \label{eq:Eq-ST}
  \mathbf{W}^{(2)}=\mathbf Q(\mathbf{W}^{(1)}).
\end{equation}
In the limit epsilon approaches
0, the dynamics are described by the equilibrium system of conservation laws
\begin{equation}
    \label{eq:HE}
    \p_t \mathbf W^{(1)} + \p_x \mathbf f^{(1)}(\mathbf W^{(1)},
    \mathbf Q(\mathbf W^{(1)})) =0.
\end{equation}
The question of the stability of the asymptotic has been analyzed 
in \cite{chen94:relax} and also in \cite{Bouchut05} where stability
conditions were exhibited. A strong stability condition is the
existence of the entropy extension : the hierarchy of models \eqref{eq:HSR}-\eqref{eq:HE}
is endowed with an entropy structure,
in the sense that the Lax  entropy-entropy flux pair of the
equilibrium system \eqref{eq:HE} extends to an entropy-entropy flux
pair for the hyperbolic system with relaxation \eqref{eq:HSR}.
More precisely, \eqref{eq:HSR} admits a convex entropy
$H: K \to \mathbb R$ such that
$\nabla^2 H(\mathbf W) \nabla \mathbf f(\mathbf W)$ is
symmetric for all $\mathbf W \in K$ and which is dissipative, that is
\begin{equation}
  \label{eq:dissip}
  \nabla H(\mathbf W) \cdot \mathbf R(\mathbf W) \leq 0, \quad \mathbf W \in K.
\end{equation}
The condition on the hessian matrix ensures the existence of an entropy flux 
$\mathbf \Psi:K\to \mathbb R^p$ such that $\nabla H(\mathbf W) \nabla \mathbf
f(\mathbf W) = \nabla \mathbf \Psi(\mathbf W)$, for all $\mathbf W \in K$,
and every strong solution to \eqref{eq:HSR} satisfies
\begin{equation}
    \label{eq:ineg-HSR}
    \p_t H(\mathbf W) + \p_x \mathbf \Psi(\mathbf W) =
    \dfrac{1}{\varepsilon}\nabla H(\mathbf W) \cdot \mathbf R(\mathbf W).
\end{equation}
The stability condition introduced by \cite{chen94:relax} states that
the restriction of the entropy pair $(H,\mathbf \Psi)$ to the
equilibrium manifold $\mathcal M_{eq}$ defines the entropy flux pair
$(\eta, \mathbf \psi)$ for the equilibrium system \eqref{eq:HE}:
\begin{equation}
    \label{eq:entrop_HE}
    \eta(\mathbf W^{(1)}) = H(E(\mathbf W^{(1)})), \quad \mathbf \psi(\mathbf W^{(1)})
    = \mathbf \Psi(E(\mathbf W^{(1)})), \quad \forall \mathbf W^{(1)}\in \mathbf M_1
    K.
\end{equation}
This strong condition implies  Liu's subcharacteristic
condition, which is weaker \cite{Bouchut05}.
The subcharacteristic condition ensures that the eigenvalues of the
relaxed system \eqref{eq:HSR} are interlaced with those of the
equilibrium system \eqref{eq:HE} in the sense that the eigenvalue
$\tilde{\lambda_i}$, $i=1,\ldots,k$, lies in the closed interval
$[\lambda_i, \lambda_{i+n-k}]$.
Hence, this interlacing maintains the correct ordering of characteristic speedsand
preventing the occurrence of nonphysical wave interactions.

\subsection{Some exemples}
\label{sec:exemples}

\subsubsection{The Jin and Xin model}
\label{sec:jin-xin-model}

The context is the one detailled in \cite{Serre2000}. We only recall the main points, as in \cite{LS01}.

Consider a system of conservation laws
\begin{equation}
  \label{eq:SCL}
   \p_t u + \p_x g(u) = 0,
 \end{equation}
 with a  nonlinear flux $g$ of class $C^2(K)$, where
 $K$ is a convex  set of admissible solutions.
We assume that this system is endowed with a entropy–entropy flux pair $(\eta,q)$.

The Jin and Xin relaxation model approximates solutions of \eqref{eq:SCL} by the relaxation system
\begin{equation}
  \label{eq:JX}
  \begin{cases}
    \p_t u + \p_x v =0\\
    \p_t v + \lambda^2 \p_x u = \dfrac{1}{\varepsilon}(g(u)-v).
  \end{cases}
\end{equation}
The wave speed $\lambda$
complies with the subcharacteristic condition
\begin{equation}
  \label{eq:subcar}
  \lambda > \max_{u\in K}\rho(\nabla_ug(u)),
\end{equation}
where $\rho(\nabla_ug(u))$ denotes the spectral radius of the jacobian of
the flux $g$.
Moreover, under the subcharacteristic
condition \eqref{eq:subcar}, the following three properties hold:
\begin{enumerate}
\item\label{it:JX-hpm} The images $K_\pm$ of $K$ under the applications $h_\pm :u
  \mapsto u\pm \dfrac 1 \lambda g(u)$ are convex sets.
\item \label{it:JX-K} $K = \dfrac 1 2 (K_+ + K_-)$,
\item\label{it:JX-Dk} The set
    $D_k^\lambda := \{ (u,v) \text{ s.t. } u+ \dfrac 1 \lambda v \in K_+
    \text{ and }  u- \dfrac 1 \lambda v \in K_-\}$
  is an invariant domain for the system \eqref{eq:JX}.
\end{enumerate}
It was proved in \cite{Serre2000} that, under the
sub-characteristic condition, an entropy–entropy flux pair $(\eta,q):K\to
\mathbb R^2$ to \eqref{eq:SCL} extends to an  entropy–entropy flux pair 
$(H,Q):D_k^\lambda\to \mathbb R^2$ to \eqref{eq:JX} which coincides
with $(\eta,q)$ on the
equilibrium manifold
$\mathcal M_{eq} = \{(u,v)\in D_k^\lambda \text{ s.t. } v=g(u)\}$.


\subsubsection{Chaplygin gas model}
\label{sec:chaplygin-gas-model}

The Chaplygin gas system presented in \cite{serre2009multidimensional}
describes the dynamics of of fluid ot covolume $\tau\in \mathbb R^*_*$
evolving with the velocity $u$. It reads
\begin{equation*}
   \begin{cases}
    \partial_t \tau - \partial_x u = 0, \\
    \partial_t u + \partial_x \big(p(\mathcal{T}) + a^2(\mathcal{T} - \tau)\big) = 0, \\
    \partial_t \mathcal{T} = \frac{1}{\epsilon}(\tau - \mathcal{T}).
\end{cases} 
\end{equation*}
with $a>0$, and $\mathcal{T}>0$. The pressure function 
p is taken, for practical applications, as the perfect-gas law $p(\mathcal{T}) =
\mathcal{T}^{-\gamma}$, with $\gamma>1$.
This model derives from Suliciu's work
\cite{suliciu1998thermodynamics}.
The eigenvalues of the system are $\lambda_1=-a$, $\lambda_2=0$, and
$\lambda_3=a$, corresponding to the characteristic wave speeds.
The equilibrium system, obtained by setting $\tau = \mathcal{T}$,
corresponds to the $p$-system:
\begin{equation*}
\begin{cases}
    \partial_t \tau - \partial_x u = 0, \\
    \partial_t u + \partial_x p(\tau) = 0.
\end{cases}
\end{equation*}
An admissible entropy for the Suliciu's system is
\begin{equation*}
  H(\tau,u,\mathcal{T})=\frac{1}{2}|u|^2+\frac{1}{1-\gamma}\mathcal{T}^{1-\gamma}+
  \frac{a^2}{2}(\mathcal{T}^2-\tau^2)+(\mathcal{T}^{-\gamma}+a^2\mathcal{T})(\tau-\mathcal{T}).
\end{equation*}
This entropy is strictly convex and dissipative with respect to the
source term under the subcharacteristic condition
\begin{equation*}
  a^2>\underset{s\in \mathbb R_+^*}{\max}(-p'(s)).
\end{equation*}

\subsubsection{Two-phase flow model}
\label{sec:two-phase-flow}

We consider a two-phase compressible 
flow model in which the two phases, indexed by $k=1,2$, are at thermal and mechanical
equilibrium and that they evolve the same velocity $u$. Mass
transfer may occur between the two phases, that are supposed to be
perfect gases in numerical applications.
We refer to \cite{hel-seg-06} for detailed computations and derivation.
The model reads as follow
\begin{equation}
  \label{eq:phase_transition_HSR}
  \begin{aligned}
    \p_t \rho  + \p_x (\rho u ) &=0,\hfill\\
    \p_t (\rho u) + \p_x (\rho u^2+p) &= 0, \hfill\\
    \p_t (\rho E) + \p_x ((\rho E+p)u) &=0,\\
    \p_t (\rho \varphi) + \p_x (\rho u \varphi) &= 
    \dfrac{\rho}{\varepsilon} (\varphi_{eq}(\rho) - \varphi),
  \end{aligned}
\end{equation}
where $\rho$ denotes the density of the mixture, $E=\demi u^2+e$ is
the total energy with $e$ the internal energy, and $\varphi\in [0,1]$ is the
mass fraction.
The mass fraction, which indicates the phase state, satisfies a
convection equation with a relaxation source term defined by
\begin{equation*}
  \varphi_{eq}(\rho)  =    \begin{cases}
    1 & \text{  if  } \rho\leq \rho_1^*,\\
     \dfrac{1/\rho - \tau_2^*}{\tau_1^* -
      \tau_2^*} &\text{  if  }  \rho_1^*\leq\rho\leq \rho_2^*,\\
    0 & \text{  if  } \rho_2^*\leq \rho,
  \end{cases}
\end{equation*}
with
\begin{equation*}
  \begin{aligned}
   \rho_1^* =  \exp(-1) \left( \dfrac{\gamma_2-1}{\gamma_1-1}
	      \right)^{\dfrac{\gamma_2}{\gamma_2-\gamma_1}},
  \quad
  \rho_2^* =  \exp(-1) \left( \dfrac{\gamma_2-1}{\gamma_1-1}
	      \right)^{\dfrac{\gamma_1}{\gamma_2-\gamma_1}}. 
  \end{aligned}
\end{equation*}
Here $\gamma_1$ and $\gamma_2$ are perfect gas coefficients.
To close the system, we use the mixture pressure law
\begin{equation*}
  p=p(\rho,e,\varphi)=(\gamma(\varphi)-1) \rho e,
\end{equation*}
with $\gamma(\varphi) = \gamma_1 \varphi +  \gamma_2 (1-\varphi)$.

 As $\varepsilon$ goes to zero, the thermodynamical equilibrium
is reached.  This asymptotic defines the equilibrium model
\begin{equation}
  \label{eq:phase_transition_HE}
  \begin{aligned}
    \p_t \rho  + \p_x (\rho u ) &=0,\hfill\\
    \p_t (\rho u) + \p_x (\rho u^2+p_{eq}) &= 0, \hfill\\
    \p_t (\rho E) + \p_x ((\rho E+p_{eq})u) &=0,
  \end{aligned}
\end{equation}
with the equilibrium pressure law introduced in \cite{hel-seg-06}
$p_{eq}= p(\rho, e, \varphi_{eq}(\rho))$ which reduces to 
\begin{equation}
  \label{eq:peq}
  p_{eq}= 
  \begin{cases}
    (\gamma_1-1)\rho e, & \text{  if }  \rho\leq \rho_1^*,\\
    (\gamma_1-1)\rho_1^* e, & \text{ if }
    \rho_1^*\leq\rho\leq \rho_2^*,\\
    (\gamma_2-1)\rho e, & \text{ if }  \rho_2^*\leq \rho.
  \end{cases}
\end{equation}

The entropy of the system \eqref{eq:phase_transition_HSR} is
  not strictly convex, see \cite{FM19} and references therein. 

\section{Staggered scheme }
\label{sec:staggered-scheme}

We present in this section a finite volume scheme which is inspired by
both the
centred scheme approaches, the so-called FORCE schemes, developed by
Toro and co-authors (see the
review \cite{ChenToro03}, and the adaptation to the two-fluid models in
\cite{TORO2020}), and the two-step staggered scheme proposed
in \cite{bereux1997roe}.
In this work, we retain the
unsplit framework of the latter reference and couple it with the FORCE approach. The
resulting scheme is consistent for any $\varepsilon$ and preserves the
desired asymptotic properties.

\subsection{Definition of the scheme}
\label{sec:definition-scheme}

Consider a piecewise constant approximation sequence $(\mathbf {W}_j^n)_{j\in
  \mathbb Z}$, where $\mathbf W_j^n$ approximates $\mathbf W(t,x)$ for all $x$
in the cell $(x_{j-1/2}, x_{j+1/2})$ of size $\Delta x$ at time
$t^n$. For simplicity, we use a uniform mesh and let $x_j$
 denote the center of  $(x_{j-1/2}, x_{j+1/2})$.
The time step $\Delta t$ satisfies the Courant-Friedrichs-Levy condition
\begin{equation}
    \label{CFLstaggered}
    \Delta t\leq \frac{\Delta x}{\underset{1\leq i\leq
        n}{\max}\lambda_i}.
  \end{equation}
\begin{figure}[ht]
    \centering
    \includegraphics[width=1\textwidth]{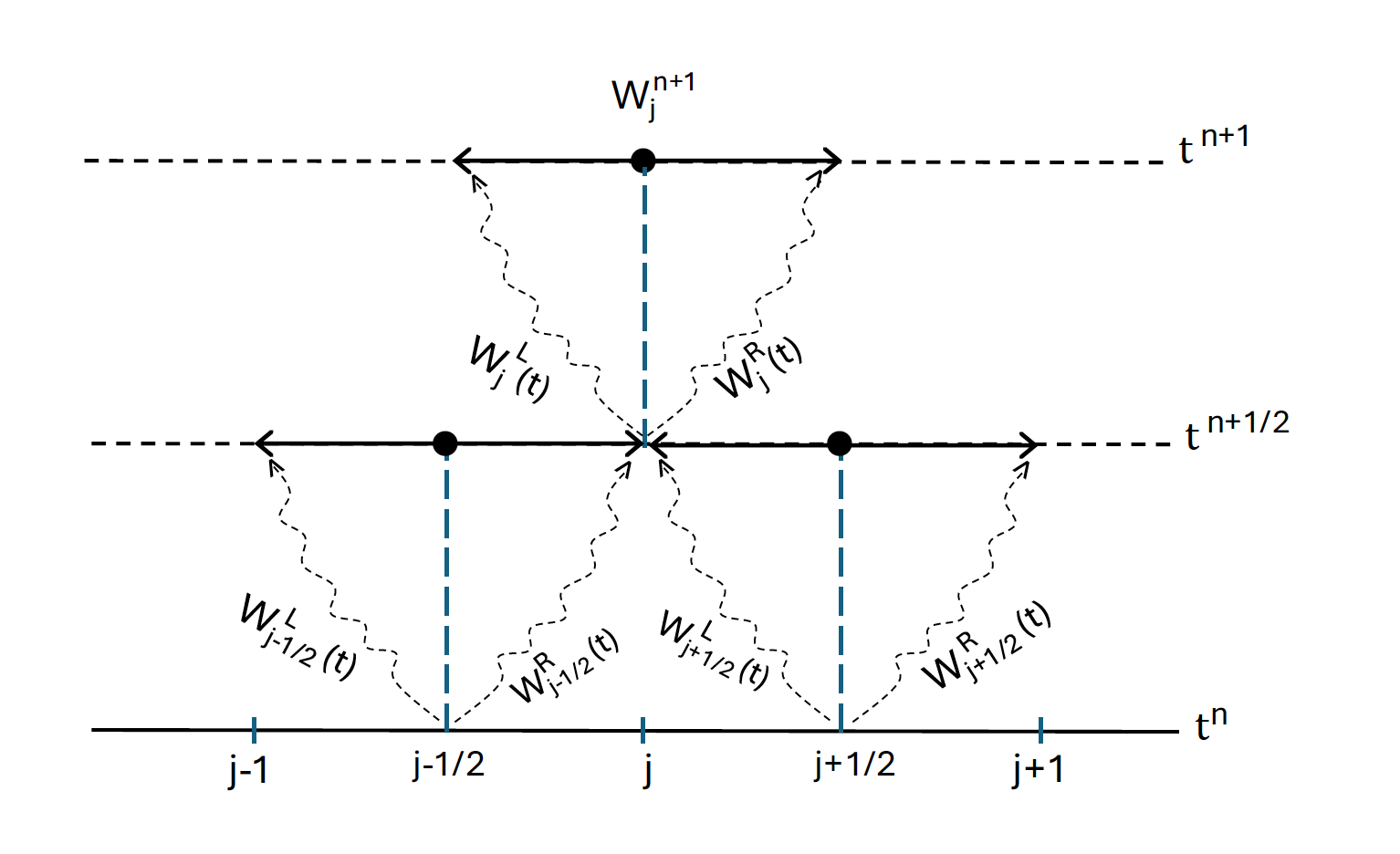}
    \label{fig:godunov}
    \caption{Illustration of the staggered scheme. The numerical
      fluxes depend on the evaluation of the source terms at each
      interfaces.}
\end{figure} 
Following \cite{bereux1997roe} the algorithm updates $\mathbf W_j^{n}$ to
a new value $\mathbf W_j^{n+1}$ in two steps.
\begin{enumerate}
\item From time $t^n$ to $t^{n+\demi}:=t^n+ \Delta t/2$
  \begin{enumerate}[label=(\Alph*)]
  \item\label{it:step1:ST} Source term time integration:
    
    Compute the solutions $\mathbf W_{j-1/2}^L(t)$ and
    $\mathbf W_{j-1/2}^R(t)$ of the following ODE systems for
    $t\in (0,\Delta t/2)$:
    \begin{equation}
      \label{eq:step1:ST}
      \begin{cases}
        \dfrac{\mathrm d }{\mathrm d t}  \mathbf W_{j-1/2}^L(t) =
        \dfrac{1}{\varepsilon} \mathbf{R}(\mathbf{W}_{j-1/2}^L(t)),\\
        \mathbf W_{j-1/2}^L(0) = \mathbf W_{j-1}^n,
      \end{cases}
      \quad
      \begin{cases}
        \dfrac{\mathrm d }{\mathrm d t} \mathbf W_{j-1/2}^R(t) =
        \dfrac{1}{\varepsilon} \mathbf{R}(\mathbf W_{j-1/2}^R(t)),\\
        \mathbf W_{j-1/2}^R(0) = \mathbf W_{j}^n.
      \end{cases}
    \end{equation}
  \item\label{it:step2:sys1} Integration of the system \eqref{eq:HSR}
    over $(t^n, t^n+ \Delta t /2)\times (x_{j-1},x_j)$, that is
    \begin{equation}
      \label{eq:step2:sys1}
      \begin{aligned}
        \mathbf W_{j-1/2}^{n+1/2} &=
         \dfrac{1}{2}\left(\mathbf W_{j}^n+\mathbf W_{j-1}^n\right)
        - \dfrac{1}{\Delta x}\int_{t^n}^{t^{n+\demi}}
        \mathbf{f}(\mathbf W(t,x_j)) \mathrm d t\\
        &+ \dfrac{1}{\Delta x}\int_{t^n}^{t^n+\demi}
        \mathbf{f}(\mathbf W(t,x_{j-1})) \mathrm d t\\
        &+ \dfrac{1}{\Delta x}\int_{x_{j-1}}^{x_j}
        \int_{t^n}^{t^{n+\demi}}
        \dfrac{1}{\varepsilon}\mathbf{R}(\mathbf{W}(t,x))\mathrm{d} t
        \mathrm{d} x,
      \end{aligned}
    \end{equation}  
    and consider the following approximations:
    \begin{itemize}
    \item Flux approximations
      \begin{equation}
        \label{eq:2step:flux}
        \begin{aligned}
          \int_{t^n}^{t^n + \demi} \mathbf f(\mathbf
          W(t,x_{j-1}))\mathrm d t &\simeq
          \dfrac{\Delta t}{2} \mathbf f \left(\mathbf
            W_{j-1/2}^L\left( \Delta t/2\right)\right),\\
          \int_{t^n}^{t^n + \demi} \mathbf f(\mathbf
          W(t,x_j))\mathrm d t &\simeq \dfrac{\Delta t}{2} \mathbf f
          \left(\mathbf W_{j-1/2}^R\left( \Delta t/2\right)\right).\\
        \end{aligned}
      \end{equation}
    \item Source term approximation
      \begin{equation}
        \label{eq:2step:ST}   
        \dfrac{1}{\Delta
          x}\int_{x_{j-1}}^{x_j}\int_{t^n}^{t^{n+\demi}}
        \dfrac{1}{\varepsilon}\mathbf R(\mathbf W(t,x))\mathrm d t \mathrm d x\simeq
        \dfrac{\Delta t}{2\varepsilon} \mathbf R \left(\mathbf W_{j-1/2}^{n+1/2}\right).
      \end{equation}
    \end{itemize}
    The previous approximations 
    lead to the following
    staggered
      approximation
      \begin{equation}
        \label{eq:sch-2step}
        \begin{aligned}
          \mathbf W_{j-1/2}^{n+1/2} &=
          \dfrac{1}{2}\left(\mathbf W_{j}^n+\mathbf W_{j-1}^n\right)
          - \dfrac{\Delta t}{2\Delta x}
          \left[\mathbf f\left( \mathbf W_{j-1/2}^R\left(
                \Delta t/2\right)\right) \right.\\
          &\left.- \mathbf f \left(\mathbf W_{j-1/2}^L
              \left( \Delta t/2\right)\right) \right] +
          \dfrac{\Delta t}{2\varepsilon} \mathbf R\left(\mathbf
            W_{j-1/2}^{n+1/2}\right).
        \end{aligned}
      \end{equation}  
  \end{enumerate}
\item\label{it:stepA+B} Repeat step \ref{it:step1:ST} and \ref{it:step2:sys1} from
  time $t^{n+\demi}$ to $t^{n+1}$ over the cell
  $[x_{j-1/2},x_{j+1/2}]$ to get the cell centered approximation of
  $\mathbf W_j^{n+1}$. 
\end{enumerate}

 At the end of step \ref{it:stepA+B}, and according to the expression \eqref{eq:sch-2step},
 the updated value at time $t^{n+1}$ actually reads
 \begin{equation}
   \label{eq:sch-final}
   \begin{aligned}
     \mathbf{W}_{j}^{n+1} &= \mathbf{W}_{j}^{n} - \dfrac{\Delta
       t}{\Delta x} \left( \tilde{\mathbf{F}}_{j+\demi} -
       \tilde{\mathbf{F}}_{j-\demi} \right) \\
     &+ \dfrac{\Delta t}{4\varepsilon} \left( 2\mathbf{R}\left(
         \mathbf{W}_{j}^{n+1} \right) + \mathbf{R}\left(
         \mathbf{W}_{j-1/2}^{n+1/2} \right) + \mathbf{R}\left(
         \mathbf{W}_{j+1/2}^{n+1/2} \right) 
     \right),
   \end{aligned}
 \end{equation}
with the numerical flux $\tilde{\mathbf{F}}_{j+\demi}$ defined
 by
 \begin{equation}
   \label{eq:flux-stag}
   \begin{aligned}
     \tilde{\mathbf{F}}_{j+\demi}&=  \dfrac{1}{4}\Bigg[\mathbf
     2\mathbf f\bigg( \mathbf W_{j}^R\Big( \Delta t/2\Big)\bigg)+\mathbf
     f\Big( \mathbf W_{j-1/2}^R\Big( \Delta t/2\Big)\Big)\Bigg.\\
     &\Bigg.+
     \mathbf f \Big(\mathbf W_{j+1/2}^R\Big( \Delta t/2\Big)\Big)-\frac{\Delta x}{\Delta t}
     \Big(\mathbf{W}_{j+1}^n-\mathbf{W}_{j}^n\Big)\Bigg].
   \end{aligned}
 \end{equation}

\subsection{Properties of the staggered scheme}
\label{sec:prop-stagg-scheme}

 While this scheme is applicable to general source term and systems of the form
 \eqref{eq:HSR}, we focus here on its implementation for the specific
 structure given by \eqref{eq:HSR2}–\eqref{eq:ST-0}.
Within this framework, the ODE systems \eqref{eq:step1:ST} can be
explicitly solved: for  $\mathbf{W}_{j-1/2}^R$, the solution is given
by
\begin{equation}
  \label{solutionEDO}
\begin{aligned}
\mathbf{W}_{j-\demi}^{(1),R}(t)&=\mathbf{W}_j^{(1),n},\\
\mathbf{W}_{j-\demi}^{(2),R}(t)&=\Big(\mathbf{W}_j^{(2),n}-
\mathbf{Q}(\mathbf{W}_j^{(1),n})\Big)e^{\frac{-t}{\varepsilon}}+\mathbf{Q}(\mathbf{W}_j^{(1),n}).
\end{aligned}
\end{equation}

It turns out that the numerical flux \eqref{eq:flux-stag} corresponds to an extension 
of the FORCE flux, applied to the states
$W_{j\pm1/2}^{L,R}\Big( t^{n+\demi}\Big)$. According to
\cite{ChenToro03}, this numerical flux corresponds
to the arithmetic average of the Lax-Friedrichs (LF) and the
Richtmyer two-step Lax-Wendroff scheme (RI) fluxes, namely
\begin{equation*}
F^{\text{FORCE}}_{j+\demi} = \dfrac{1}{2} \left(
  F^{\text{RI}}_{j+\demi} + F^{\text{LF}}_{j+\demi} \right).
\end{equation*}
The following asymptotic preserving property relies on this analogy.

\begin{Prop}[Asymptotic preserving property]
  Let the constant sequence of cell-averaged values
  $(\mathbf{W}_j^{(1),n},\mathbf{W}_j^{(2),n})$ be given at time
  $t^n$, for $j\in \mathbb Z$.
  Under the CFL condition \eqref{CFLstaggered},
  the scheme \eqref{eq:sch-final} is asymptotic preserving, in the
  sense that it is consistent with solutions of the hyperbolic model
  \eqref{eq:HSR} for all $\varepsilon>0$ and,
  in the limit $\varepsilon\to 0$, it converges to the stable and
  consistent FORCE scheme for the hyperbolic equilibrium model
  \eqref{eq:HE}.
\end{Prop}

\begin{proof}
  We first address consistency by evaluating $\tilde{\mathbf{F}}_{j+\demi}(\mathbf W,\mathbf W)$ for any $\varepsilon>0$ and $\mathbf W\in K$.
  In \eqref{eq:flux-stag}, it holds
  \begin{equation*}
  \mathbf W_{j}^R\Big( \Delta t/2\Big)
  =\mathbf W_{j-\demi}^R\Big( \Delta t/2\Big)
  =\mathbf W_{j+\demi}^R\Big( \Delta t/2\Big)=\mathbf W.
\end{equation*}  
  Hence $\tilde{\mathbf{F}}_{j+\demi}(\mathbf W, \mathbf W)
  = \mathbf f(\mathbf W)$.
  In particular, if the source term
  vanishes, the numerical scheme reduces to the standard FORCE
  scheme for the homogeneous hyperbolic system associated with \eqref{eq:HSR2}.

  We now establish the asymptotic preserving property in the limit
  $\varepsilon\to 0$, focusing on source terms of the form
  \eqref{eq:ST-0}. At equilibrium, \eqref{eq:sch-2step} yields $\mathbf R(\mathbf W_{j-\demi}^{\,n+\demi})=0$, and using \eqref{eq:ST-0} we obtain
   \begin{equation*}
    \mathbf{W}_{j\pm\demi}^{(2),n+\demi}=\mathbf{Q}\Big(\mathbf{W}_{j\pm\demi}^{(1),n+\demi}\Big).
  \end{equation*}
  By substituting the solution of the ODE step \eqref{solutionEDO} into \eqref{eq:sch-final}, we obtain
  the following consistent scheme for the
  equilibrium model \eqref{eq:HE} which reads
  \begin{equation}
    \label{eq:sch-final0} 
    \mathbf{W}_{j}^{(1),n+1} = \mathbf{W}_{j}^{(1),n} - \dfrac{\Delta
      t}{\Delta x} \left( \tilde{\mathbf{F}}^{(1)}_{0,j+\demi} -
      \tilde{\mathbf{F}}^{(1)}_{0,j-\demi} \right),
  \end{equation}
where the numerical flux is the FORCE flux \cite{ChenToro03},
\begin{equation}
\begin{aligned}
\tilde{\mathbf{F}}^{(1)}_{0,j+\demi} &= \dfrac{1}{4} \Bigg[
    2\mathbf{f}^{(1)}\left(\mathbf{W}_{j+\demi}^{(1),n+\demi},
      \mathbf{Q}\left(\mathbf{W}_{j+\demi}^{(1),n+\demi}\right)\right)
    \\
    &\quad + 
        \mathbf{f}^{(1)}\left(\mathbf{W}_{j+1}^{(1),n},
          \mathbf{Q}\left(\mathbf{W}_{j+1}^{(1),n}\right)\right)
        + \mathbf{f}^{(1)}\left(\mathbf{W}_{j}^{(1),n},
          \mathbf{Q}\left(\mathbf{W}_{j}^{(1),n}\right)
    \right) \\
    &\quad - \dfrac{\Delta x}{\Delta t}
    \left( \mathbf{W}_{j+1}^{(1),n} - \mathbf{W}_{j}^{(1),n} \right)
    \Bigg].
\end{aligned}
\end{equation}
\end{proof}
It was proved in \cite{ChenToro03} that the FORCE scheme is consistent with
the Lax entropy inequality for hyperbolic systems of conservation laws
of the form \eqref{eq:HE}, in the sense that the finite volume scheme
satisfies  a \textit{global} discrete entropy inequality
\begin{equation*}
  \sum_{j\in \mathbb Z}\dfrac{\eta(\mathbf W_j^{n+1})-\eta(\mathbf
    W_j^{n})}{\Delta t} \Delta x\leq 0,
\end{equation*}
where $\eta$ is the entropy for the equilibrium system \eqref{eq:HE}.
Since the scheme \eqref{eq:sch-final}–\eqref{eq:flux-stag} reduces, as $\varepsilon\to0$, to the FORCE
scheme applied to the equilibrium system \eqref{eq:HE}, and, as $\varepsilon\to\infty$, to the FORCE
scheme for the homogeneous hyperbolic system associated with \eqref{eq:HSR2}, the staggered scheme
\eqref{eq:sch-final} satisfies a global version of discrete entropy inequality in these two regimes.

Because the numerical fluxes \eqref{eq:flux-stag} depend on the exact solutions of the ODE step
\eqref{eq:step1:ST}, $L^\infty$ stability and invariant-domain preservation are not immediate.
For the Jin–Xin model \eqref{eq:JX}, however, one can prove that the staggered scheme
\eqref{eq:sch-final} preserves the invariant domain $K$.

Applying \eqref{eq:sch-2step} and \eqref{eq:sch-final} to the Jin and
Xin model \eqref{eq:JX} gives
\begin{equation}
\label{eq:JX-stag}
\begin{aligned}
u_{j-\demi}^{n+\demi}&= \frac{u_{j-1}^n+u_{j}^n}{2}-\frac{\Delta
  t}{2\Delta x}\bigg((v_{j}^n-g(u_{j}^n))e^{-\frac{\Delta
    t}{2\epsilon}}+g(u_{j}^n)\\
&\qquad \qquad-(v_{j-1}^n-g(u_{j-1}^n))e^{-\frac{\Delta
    t}{2\epsilon}}-g(u_{j-1}^n)\bigg),\\
v_{j-\demi}^{n+\demi}&=\left(\frac{1}{1+\frac{\Delta
      t}{2\epsilon}}\right)\left(\frac{v_{j-1}^n+v_{j}^n}{2}-\frac{\lambda^2\Delta
    t}{2\Delta x}(u_j^n-u_{j-1}^n)+\frac{\Delta
    t}{2\epsilon}g(u_{j-\demi}^{n+\demi})\right),\\
u_{j}^{n+1}&= \frac{u_{j-\demi}^{n+\demi}+u_{j+\demi}^{n+\demi}}{2} - \frac{\Delta t}{2 \Delta
  x}\bigg((v_{j+\demi}^{n+\demi}-g(u_{j+\demi}^{n+\demi}))e^{-\frac{\Delta
    t}{2\epsilon}}+g(u_{j+\demi}^{n+\demi})\\
&\qquad \qquad
-(v_{j-\demi}^{n+\demi}-g(u_{j-\demi}^{n+\demi}))e^{-\frac{\Delta
    t}{2\epsilon}}-g(u_{j-\demi}^{n+\demi})\bigg)\\
v_{j}^{n+1}&= \left( \frac{1}{1 + \frac{\Delta t}{2\epsilon}} \right)
\left( \frac{v_{j-\demi}^{n+\demi}+v_{j+\demi}^{n+\demi}}{2} - \frac{\Delta t \lambda^2}{2 \Delta
    x}(u_{j+\demi}^{n+\demi} - u_{j-\demi}^{n+\demi}) + \frac{\Delta
    t}{2 \epsilon}g(u_j^{n+1}) \right).
\end{aligned}
\end{equation}

In order to prove that the scheme preserves the invariant domain \(K\),
we prove the $L^\infty$ stability property of the first two steps
\ref{it:step1:ST}-\ref{it:step2:sys1}
of the
algorithm. The proofs are based on the symmetric variables

\begin{equation}
  \label{eq:JX-rs}
  r = u+ \dfrac{1}{\lambda}v, \qquad  s= u- \dfrac{1}{\lambda}v,
\end{equation}
and on the maps \(h_\pm(u) := u \pm \frac{1}{\lambda} g(u)\) introduced above.

\begin{Prop}
  \label{prop:JX-Stag-step2-Linf}
  If $(u_j^n,v_j^n)\in
  D_{\mathcal{K}}^{\lambda}$, for all $j\in\mathbb{Z}$, then
  the solutions $r_{j+1/2}^R(t)=u_{j+1/2}^R(t)+\dfrac 1 \lambda
  v_{j+1/2}^R(t)$ and
  $s_{j+1/2}^R(t)=u_{j+1/2}^R(t)-\dfrac 1 \lambda v_{j+1/2}^R(t)$ associated to the Cauchy problems
  \eqref{eq:step1:ST} belong to $K_+$ and $K_-$ respectively, for all $t>0$.

  Moreover, under the subcharacteristic condition \eqref{eq:subcar} and the CFL
  condition \eqref{CFLstaggered}, if $(u_j^n,v_j^n)\in
  D_{\mathcal{K}}^{\lambda}$, for all $j\in\mathbb{Z}$, then
  $u_{j-\demi}^{n+\demi}$ belongs to $K$,
  for
  all $ j\in\mathbb{Z}$.
\end{Prop}

\begin{proof}
Using \eqref{solutionEDO} and the definition of $r_{j+1/2}^R(t)$ lead
to
\begin{equation*}
  \begin{aligned}
    r_{j+1/2}^R(t) &=u_{j+1/2}^R(t)+\dfrac 1 \lambda v_{j+1/2}^R(t)\\
    &= u_{j+1}^n  + \dfrac{1}{\lambda} \left( (v_{j+1}^n -
      g(u_{j+1}^n)) e^{-t/\varepsilon} + g(u_{j+1}^n) \right)\\
    &= u_{j+1}^n (1-e^{- t/\varepsilon}) + e^{-
      t/\varepsilon} (u_{j+1}^n +\dfrac{1}{\lambda} v_{j+1}^n) +
    \dfrac{1}{\lambda} (1-e^{- t/\varepsilon}) g(u_{j+1}^n)\\
    &= (1-e^{- t/\varepsilon}) h_+(u_{j+1}^n) + e^{-
      t/\varepsilon} r_{j+1}^n,
  \end{aligned}
\end{equation*}
  where we used the definition of $ r_{j+1}^n$ and 
  of $h_+(u)$ (see Section~\ref{sec:jin-xin-model}, Item~\ref{it:JX-hpm}).
  Since $u_{j+1}^n\in K$, we have $h_+(u_{j+1}^n)\in K_+$, and $r_{j+1}^n\in K_+$ as well. Hence $r_{j+1/2}^R(t)$ is a convex combination
  of elements of the convex set $K_+$ and thus belongs to $K_+$. The same arguments
  yield $r_{j\pm 1/2}^L(t)\in K_+$ and $s_{j\pm1/2}^{R,L}(t)\in
  K_-$ for all $j\in \mathbb Z$ and $t>0$.
  
 We now consider the numerical scheme \eqref{eq:JX-stag}. At time $t^{n+\demi}$, it
  rewrites
  \begin{align*}
    u_{j-\demi}^{n+\demi}&= \dfrac{1}{2}\left( u_{j-\demi}^L(\Delta t/2)
                           + u_{j-\demi}^R(\Delta t/2)  \right)
                           - \dfrac{\Delta t}{2 \Delta x}\left(
                           v_{j-\demi}^R(\Delta t/2) -
                           v_{j-\demi}^L(\Delta t/2)\right).
  \end{align*}
  Adding and subtracting $\lambda u_{j-\demi}^L(\Delta t/2)$ and $\lambda u_{j-\demi}^R(\Delta t/2)$ in the second term gives
  \begin{align*}
    u_{j-\demi}^{n+\demi}&= \left( \dfrac{1}{2} - \dfrac{\lambda
                           \Delta t}{2\Delta x}\right)
                           u_{j-\demi}^L(\Delta t/2)
                           + \left( \dfrac{1}{2} - \dfrac{\lambda
                           \Delta t}{2\Delta x}\right)
                           u_{j-\demi}^R(\Delta t/2)\\
                           &+ \dfrac{\lambda \Delta t}{\Delta x}
                           \left( \dfrac 1 2 s_{j-\demi}(\Delta t/2)
                           + \dfrac 1 2 r_{j-\demi}(\Delta t/2)\right).
  \end{align*}
  The first two terms belong to $K$ and the coefficients are positive
  under the CFL condition \eqref{CFLstaggered}. 

  By the previous result, $s_{j-\demi}(\Delta t/2)\in K_-$ and
  $r_{j-\demi}(\Delta t/2)\in K_+$. Since $\dfrac 1 2 K_-+ \dfrac 1 2
  K_+=K$, the weighted sum of the  two
  last terms belong to $K$. Therefore $u_{j-\demi}^{n+\demi}$ is a convex
  combination of elements of the convex set $K$ which concludes the proof.
\end{proof}

\section{An approximate Riemann solver accounting for the source term}
\label{sec:ARS}

In this Section, we propose an approximate Riemann solver which takes
into account the source term. We follow the methodology provided in
\cite{berthon2016well}, originally developed for mixed
hyperbolic/parabolic system of partial differential equations with a
source term involving spatial derivatives.
Introducing an approximate Riemann solver, the final scheme is shown to
be well balancing, capturing steady-state equilibria. However,
the presence of the source term introduces challenges in accurately
computing the mean value of the exact Riemann solution of the
relaxation system \eqref{eq:HSR2}
we are interested in.

\subsection{Definition of the scheme}
\label{sec:def-ARS}
To derive the numerical scheme, we introduce an approximate Riemann
solver in the sense of Harten, Lax and van Leer
\cite{harten1983upstream}.
An approximate Riemann solver $\tilde{\mathbf W}(x/t; \mathbf W_\ell,
\mathbf W_r)$ is a self similar function that reproduces the exact
solution  ${\mathcal W}_{\mathcal R}(x,t; \mathbf W_\ell, \mathbf W_r)$
of the Riemann problem of \eqref{eq:HSR} with an initial data
\begin{equation*}
    \textbf{W}(0, x)= \begin{cases}
        \textbf{W}_\ell \quad \text{if} \quad x < 0, \\
        \textbf{W}_r \quad \text{if} \quad x > 0,
    \end{cases}
\end{equation*}
where $\textbf{W}_\ell$ and $\textbf{W}_r$ are two given constant states. 
Here we consider an approximate Riemann solver with three constant
states separated by speeds \(\lambda_\ell<0<\lambda_r\), namely
\begin{equation}
\label{eq:ARS}
\widetilde{\textbf{W}}(x/t,\textbf{W}_\ell,\textbf{W}_r)=\begin{cases}
     \textbf{W}_\ell \quad\text{if}\quad \frac{x}{ t} < \lambda_\ell,\\
   \textbf{W}^* \quad\text{if}\quad  \lambda_\ell < \frac{x}{ t} < \lambda_r, \\
    \textbf{W}_r \quad\text{if}\quad \frac{x}{ t} > \lambda_r,
\end{cases}
\end{equation}
where $\lambda_{\ell,r}$ are chosen large enough to ensure robustness
\cite{harten1983upstream}.
In order to determine the intermediate state, the integral consistency
condition \cite{harten1983upstream} must be fulfilled, in the sense that
$\widetilde{\textbf{W}}$
must satisfy:
\begin{equation}
    \label{eq:ConsistencyCondition}
    \dfrac{1}{\Delta x} \int_{-\Delta x/2}^{\Delta
      x/2}\widetilde{\textbf{W}}(x/\Delta t; \mathbf W_\ell,
    \mathbf W_r) \mathrm d x
    = \dfrac{1}{\Delta x}  \int_{-\Delta x/2}^{\Delta x/2}{\mathcal
      W}_{\mathcal R}(x,\Delta t; \mathbf W_\ell, \mathbf W_r) \mathrm d
    x.
\end{equation}
According to \eqref{eq:ARS}, the left-hand-side of
\eqref{eq:ConsistencyCondition} reads
\begin{equation}
    \label{eq:ARS-1}
    \begin{aligned}
    \dfrac{1}{\Delta x} \int_{-\Delta x/2}^{\Delta
      x/2}\widetilde{\textbf{W}}(x/\Delta t; \mathbf W_\ell,
    \mathbf W_r) \mathrm d x
    &= \dfrac{1}{2}(\mathbf W_\ell+\mathbf W_r) \\
    &+ \dfrac{\Delta t}{\Delta x} (\lambda_\ell \mathbf W_\ell -
    \lambda_r \mathbf W_r)
    + \dfrac{\Delta t}{\Delta x} (\lambda_r-\lambda_\ell) \mathbf W^*.
    \end{aligned}
\end{equation}

The objective is now to provide an accurate evaluation of the average
of the exact Riemann solver ${\mathcal W}_{\mathcal R}(x, t; \mathbf
W_\ell, \mathbf W_r) $.
A closed-form evaluation
is out of reach because of the relaxation source term, so in \eqref{eq:ConsistencyCondition}
we replace the exact solution by a suitable approximation.
To compute the right-hand side of \eqref{eq:ConsistencyCondition}, we integrate
\eqref{eq:HSR} over the space-time domain $(-\tfrac{\Delta x}{2},\,\tfrac{\Delta x}{2})\times(0,\Delta t)$ to get
\begin{equation}
\label{eq:exact-RS}
\begin{aligned}
&\dfrac{1}{\Delta x}\int_{\frac{-\Delta x}{2}}^{\frac{\Delta
    x}{2}}{\mathcal W}_{\mathcal R}(x, \Delta t ; \mathbf W_\ell, \mathbf
W_r)\mathrm d x\\
& =\dfrac{1}{2}(\mathbf W_\ell +\mathbf W_r)-\dfrac{1}{\Delta
  x}\int_0^{\Delta t}\mathbf{f}({\mathcal W}_{\mathcal
  R}(\dfrac{\Delta x}{2},t; \mathbf W_\ell, \mathbf W_r))\mathrm d t\\
&+\dfrac{1}{\Delta x}\int_0^{\Delta t}
\mathbf{f}({\mathcal W}_{\mathcal R}(-\dfrac{\Delta x}{2},t; \mathbf
W_\ell, \mathbf W_r))\mathrm d t\\
&+\dfrac{1}{\varepsilon}\frac{1}{\Delta x}\int_0^{\Delta
  t}\int_{\frac{-\Delta x}{2}}^{\frac{\Delta
    x}{2}}\mathbf{R}({\mathcal W}_{\mathcal R}(x, t ; \mathbf
W_\ell, \mathbf W_r))\mathrm d x \mathrm d t.
\end{aligned}
\end{equation}
Because of the source term, we may fear that 
\begin{equation}
    {\mathcal W}_{\mathcal R}(-\Delta x/2,  t ; \mathbf W_\ell,
    \mathbf W_r)) \neq \mathbf W_\ell,
    \quad
    {\mathcal W}_{\mathcal R}(\Delta x/2,  t  ; \mathbf W_\ell,
    \mathbf W_r)) \neq \mathbf W_r.
\end{equation}
Therefore we adopt the following approximations
\begin{equation}
    \label{eq:ARS-flux}
    \begin{aligned}
    &\int_0^{\Delta t}\mathbf{f}({\mathcal W}_{\mathcal
      R}(-\Delta x/2,t; \mathbf W_\ell, \mathbf W_r) \simeq
    \Delta t \mathbf f(\mathbf W^L(\Delta t)),\\
    &\int_0^{\Delta t}\mathbf{f}({\mathcal W}_{\mathcal R}
    (\Delta x/2,t; \mathbf W_\ell, \mathbf W_r)
    \simeq \Delta t \mathbf f(\mathbf W^R(\Delta t)),
    \end{aligned}
\end{equation}
where the states $\mathbf W^L(\Delta t)$ and $\mathbf W^R(\Delta t)$
are solutions to the following ODE systems
  \begin{equation}
  \label{eq:ST-int}
    \begin{cases}
      \dfrac{\mathrm d}{\mathrm dt}\mathbf{W}^L(t) =
    \frac{1}{\epsilon} \mathbf{R}(\mathbf{W}^L(t)),\\
    \mathbf{W}^L(0) = \mathbf{W}_\ell,
  \end{cases}
  \quad
  \begin{cases}
    \dfrac{\mathrm d}{\mathrm dt}\mathbf{W}^R(t) =
    \frac{1}{\epsilon} \mathbf{R}(\mathbf{W}^R(t)), \text{ for } t>0,\\
    \mathbf{W}^R(0) = \mathbf{W}_r.
  \end{cases}
\end{equation}
Here we make use of the approximation
by Béreux and Sainsaulieu \cite{bereux1997roe}, previously used in the
staggered scheme of Section \ref{sec:staggered-scheme}.
The source term is incorporated into the flux approximation, which is then used in the construction of the approximate Riemann solver.

Concerning the source term, we substitute the source term integral by
a consistent approximation
\begin{equation}
    \label{eq:ARS-source}
    \dfrac{1}{\Delta x}\int_{\frac{-\Delta x}{2}}^{\frac{\Delta
        x}{2}}\mathbf{R}({\mathcal W}_{\mathcal R}(x, \Delta t ;
    \mathbf W_\ell, \mathbf W_r))\mathrm d x
    \simeq \{\mathbf R(\mathcal{W}_{\mathcal{R}})\}(\Delta x, \Delta
    t; \mathbf W_\ell, \mathbf W_r).
\end{equation}
The definition of the solver depends strongly on the structure of the source term.
Here, we propose a method adapted to systems of type \eqref{eq:HSR2} with source terms of
the form \eqref{eq:ST-0}.
The source term approximation
$\{\mathbf R(\mathcal{W}_{\mathcal{R}})\}$ is then given by
\begin{equation}
    \label{eq:ARS-source2}
    \begin{aligned}
    \dfrac{1}{\Delta x}\int_{\frac{-\Delta x}{2}}^{\frac{\Delta
        x}{2}}\mathbf{R}({\mathcal W}_{\mathcal R}(x, \Delta t ;
    \mathbf W_\ell, \mathbf W_r))\mathrm d x
    &\simeq \{\mathbf Q\}(\Delta x, \Delta t; \mathbf W_\ell, \mathbf W_r)\\
    &- \dfrac{1}{\Delta x}\int_{\frac{-\Delta x}{2}}^{\frac{\Delta
        x}{2}} {\mathcal W}_{\mathcal R}^{(2)}(x, \Delta t ; \mathbf
    W_\ell, \mathbf W_r)dx,
    \end{aligned}
\end{equation}
where ${\mathcal W}_{\mathcal R}^{(2)}(x, \Delta t ; \mathbf W_\ell,
\mathbf W_r)$
denotes the last $n-k$ components of the exact Riemann solver.
The term $\{\mathbf Q\}(\Delta x, \Delta t; \mathbf W_\ell, \mathbf
W_r)$ refers to the actual source term approximation and it will be defined later on.
For the sake of readability, this term is denoted $\{\mathbf Q\}_{\ell,r}$.

Using the approximations \eqref{eq:ARS-flux} and
\eqref{eq:ARS-source2}, we can propose an approximation of
the exact Riemann solution, which is denoted $\tilde{\mathcal
  W}_{\mathcal R}(x,  t ; \mathbf W_\ell, \mathbf W_r)$.
Since the source term acts only on the last $n-k$ components of the
vector $\tilde{\mathcal W}_{\mathcal R}$, we split the contributions
and write
$\tilde{\mathcal W}_{\mathcal R} = (\tilde{\mathcal W}_{\mathcal
  R}^{(1)}, \tilde{\mathcal W}_{\mathcal R}^{(2)})$.

\begin{Lem}
\label{lem:ERS}
Let us assume that the approximation $\{\mathbf Q\}_{\ell,r}$ does not depend on its second argument and that \eqref{condition jacob} holds.
Then the approximation of the exact Riemann solver is defined by

\begin{equation}
    \begin{aligned}
    \label{eq:W1}
    \dfrac{1}{\Delta x}\int_{\frac{-\Delta x}{2}}^{\frac{\Delta
        x}{2}}\tilde{\mathcal W}_{\mathcal R}^{(1)}(x,  t ; \mathbf
    W_\ell, \mathbf W_r) \mathrm d x &=
    \dfrac{1}{2}(\mathbf W^{(1)}_\ell + \mathbf W^{(1)}_r) \\
    &- \dfrac{\Delta t}{\Delta x}(\mathbf f^{(1)}
    (\mathbf W^R( \Delta t)) - \mathbf f^{(1)}(\mathbf W^L(\Delta
    t))),
    \end{aligned}   
\end{equation}
\begin{equation}
    \begin{aligned}
    \label{eq:W2}
        \dfrac{1}{\Delta x}\int_{\frac{-\Delta x}{2}}^{\frac{\Delta
            x}{2}}\tilde{\mathcal W}_{\mathcal R}^{(2)}(x,  t ;
        \mathbf W_\ell, \mathbf W_r) \mathrm d x &=
    \dfrac{1}{2}(\mathbf W^{(2)}_\ell + \mathbf W^{(2)}_r) e^{-\Delta
      t/\varepsilon}\\
    &- \dfrac{\varepsilon}{\Delta x}(1-e^{-\Delta
      t/\varepsilon})(\mathbf f^{(2)}(\mathbf W^R(\Delta t)) - \mathbf
    f^{(2)}(\mathbf W^L(\Delta t)))\\
    &+(1-e^{-\Delta t/\varepsilon}) \{\mathbf Q\}_{\ell,r}.
    \end{aligned}   
\end{equation}
\end{Lem}

\begin{proof}
For the $k$ first components, one easily observes that the exact Riemann
solver satisfies \eqref{eq:W1} since the source term has no contribution.
We now focus on the last $n-k$ components. We introduce the smooth function
\begin{equation}
\label{eq:approx solution}
\mathcal{F}( t)=\dfrac{1}{\Delta x}\int_{\frac{-\Delta x}{2}}^{\frac{\Delta x}{2}}
\tilde{\mathcal W}_{\mathcal R}^{(2)}(x,  t ; \mathbf W_\ell, \mathbf
W_r) \mathrm d x.
\end{equation}
Plugging the notation \eqref{eq:approx solution} into 
\eqref{eq:exact-RS} and using the
approximations \eqref{eq:ARS-flux} and \eqref{eq:ARS-source2}  gives
\begin{equation}
\label{eq:F(Delta t)}
    \begin{aligned}
  \mathcal{F}(\Delta t)&=
  \dfrac{1}{2}(\mathbf{W}^{(2)}_\ell+\mathbf{W}_r^{(2)})-\dfrac{\Delta
    t}{\Delta x}\bigg(\mathbf{f}^{(2)}(\mathbf{W}^R(\Delta
  t))-\mathbf{f}^{(2)}(\mathbf{W}^L(\Delta t))\bigg)\\
  &+\dfrac{1}{\varepsilon}\Delta t\{\mathbf Q\}_{\ell,r}-
  \frac{1}{\varepsilon}\int_0^{\Delta t}\mathcal{F}(t)\mathrm d t.
      \end{aligned}
\end{equation}
Since $\{\mathbf Q\}_{\ell,r}$ does not depend on $\Delta t$, and
assuming \eqref{condition jacob}, the derivative reads
\begin{equation}
\label{eq:Fprime}
\mathcal{F}'(\Delta t)+\frac{1}{\varepsilon}\mathcal{F}(\Delta
t)=-\dfrac{1}{\Delta x}\bigg(\mathbf{f}^{(2)}(\mathbf{W}^R(\Delta
t))-\mathbf{f}^{(2)}(\mathbf{W}^L(\Delta t) )\bigg)
+\dfrac{1}{\varepsilon}\{\mathbf Q\}_{\ell,r}.
\end{equation}
Moreover, by \eqref{condition jacob}, $\frac{d}{dt}\,\mathbf f^{(2)}(\mathbf W^{R,L}(t))=0$,
hence $\mathbf{f}^{(2)}(\mathbf{W}^{R,L}(t))\equiv \mathbf{f}^{(2)}(\mathbf{W}_{r,\ell})$
for $t\in[0,\Delta t]$. Solving \eqref{eq:Fprime} with the initial condition
$\mathcal F(0)=\tfrac{1}{2}\big(\mathbf W_\ell^{(2)}+\mathbf W_r^{(2)}\big)$ gives
 \begin{equation}
 \label{eq:sol_F}
 \begin{aligned}
 \mathcal{F}(\Delta
 t)=&\Bigg(\frac{\mathbf{W}^{(2)}_\ell+\mathbf{W}^{(2)}_r}{2}+\dfrac{\varepsilon}{\Delta
   x}\bigg(\mathbf{f}^{(2)}(\mathbf{W}^R(\Delta
 t))-\mathbf{f}^{(2)}(\mathbf{W}^L(\Delta t))\bigg)\Bigg.\\
 &\qquad \qquad \qquad \Bigg.-\{\mathbf Q\}_{\ell,r}\Bigg)e^{\frac{-\Delta t}{\varepsilon}}\\
 &-\dfrac{\varepsilon}{\Delta
   x}\bigg(\mathbf{f}^{(2)}(\mathbf{W}^R(\Delta t))-\mathbf{f}^{(2)}(\mathbf{W}^L(\Delta t))\bigg)+\{\mathbf
 Q\}_{\ell,r},
 \end{aligned}
  \end{equation}
This is precisely the expression stated in \eqref{eq:W2}.
\end{proof}

\begin{Rem}
If the model does not satisfy \eqref{condition jacob}, then
$\mathbf f^{(2)}(\mathbf W^{L,R}(t))$ is not constant during the source step.
In that case, we directly use \eqref{eq:W2} to approximate the exact Riemann solution.
\end{Rem}

With the approximation of the exact Riemann solution $\tilde{\mathcal
  W}_{\mathcal R}(x,  t ; \mathbf W_\ell, \mathbf W_r)$,
given in Lemma \ref{lem:ERS},
we now construct the approximate Riemann solver.
Instead of enforcing the classical integral consistency condition
\eqref{eq:ConsistencyCondition}, we impose that
\begin{equation}
    \label{eq:ApproxConsistencyCondidiotn}
\dfrac{1}{\Delta x}\int_{\frac{-\Delta x}{2}}^{\frac{\Delta x}{2}}
\tilde{\textbf{W}} \left(x/\Delta t; \textbf{W}_\ell,
  \textbf{W}_r\right) \mathrm d x = \frac{1}{\Delta
  x}\int_{\frac{-\Delta x}{2}}^{\frac{\Delta x}{2}}
\tilde{\mathcal{W}}_{\mathcal{R}} \left( x, \Delta t ; \textbf{W}_\ell,
  \textbf{W}_r\right) \mathrm d x.
\end{equation}
By identification, we determine the intermediate state $\mathbf W^*$
of the approximate Riemann solver which reads
\begin{eqnarray}
  \label{eq:W1*}
\mathbf W^{(1),*} &=&-\dfrac{1}{\lambda_r-\lambda_\ell}\bigg(\mathbf
f^{(1)}(\mathbf W^R(\Delta t)) - \mathbf f^{(1)}(\mathbf W^L(\Delta
t))\bigg)\\
    &\quad&+ \dfrac{1}{\lambda_r-\lambda_\ell}\bigg(\lambda_r \mathbf W^{(1)}_r
    - \lambda_\ell \mathbf W^{(1)}_\ell\bigg),\nonumber\\
%
\label{eq:W2*}
\mathbf W^{(2),*} &=&\dfrac{\Delta x }{\Delta t (\lambda_r -
  \lambda_\ell)}\bigg(e^{-\Delta t/\varepsilon}- 1\bigg)
\frac{\mathbf{W}^{(2)}_\ell+\mathbf{W}^{(2)}_r}{2}\\
&\quad&+ \bigg(e^{-\Delta t/\varepsilon}- 1\bigg)\dfrac{\varepsilon }{\Delta
  t (\lambda_r - \lambda_\ell)}\bigg(\mathbf f^{(2)}(\mathbf W^R(\Delta
t)) - \mathbf f^{(2)}(\mathbf W^L(\Delta t))\bigg) \nonumber\\
&\quad&- \dfrac{(\lambda_\ell\mathbf W^{(2)}_\ell
  - \lambda_r \mathbf W^{(2)}_r)}{\lambda_r-\lambda_\ell} 
- \bigg(e^{-\Delta t/\varepsilon}- 1\bigg)
\dfrac{\Delta x }{\Delta t (\lambda_r - \lambda_\ell)} \{\mathbf
Q\}_{\ell,r}. \nonumber
\end{eqnarray}
We now need to specify $\{\mathbf Q\}_{\ell,r}$
in order to guarantee the asymptotic preserving property as
$\varepsilon$ tends to 0.

\begin{figure}[h!]
    \centering
    \includegraphics[width=1\textwidth]{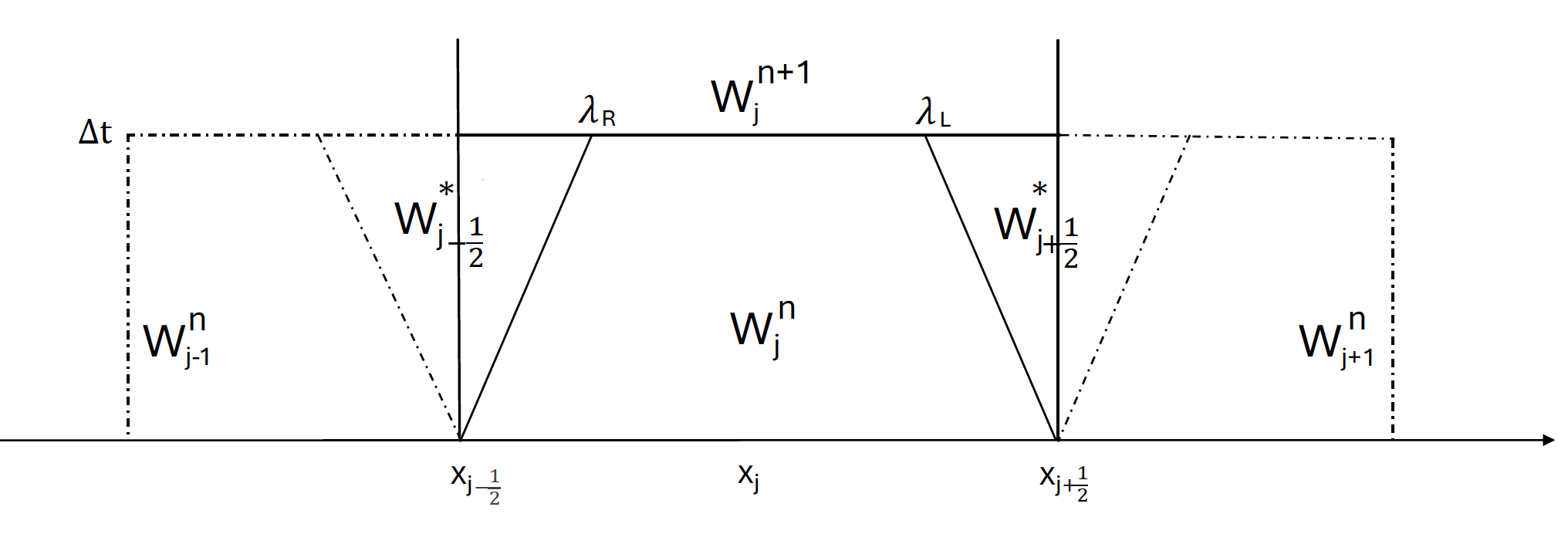}
    \caption{Juxtaposition of the approximate Riemann solvers defining
    the sequence $\mathbf W_j^{n+1}$, $j\in \mathbb Z$, using the
    intermediate states $\mathbf W_{j\pm\demi}^*$.}\label{fig:godunov2}
\end{figure}
Equipped with the approximate Riemann solver $\tilde{\mathbf{W}}$,
defined by \eqref{eq:ARS}, \eqref{eq:W1*} and \eqref{eq:W2*} ,
we derive a Godunov-type finite-volume scheme.
Figure \ref{fig:godunov2} illustrates the standard configuration, 
where the mesh and time-discretization notations are those introduced at the beginning of Section~\ref{sec:staggered-scheme}.

The initial datum is constant within each cell 
\begin{equation}
    \label{eq:wi0}
    \mathbf W_j^0 = \dfrac{1}{\Delta x} \int_{-\Delta x/2}^{\Delta
      x/2}\mathbf{W}(0,x)\mathrm d x.
\end{equation}
The time step is constrained by the CFL condition: 
\begin{equation}\label{CFL:ARS}
    \frac{\Delta t}{\Delta x}
    \underset{j\in \mathcal{Z}}{\max}
    \bigg(|\lambda_{\ell,j+1/2}|,|\lambda_{r,j+1/2}|\bigg) \leq \frac{1}{2},
\end{equation}
where $\lambda_{\ell,r,j+1/2}$ represent the left and right wave
velocities associated with the approximate Riemann solver at the
interface $x_{j+1/2}$.

The updated state
$\mathbf{W}_{j}^{n+1}=(\mathbf{W}_{j}^{(1),n+1},\mathbf{W}_{j}^{(2),n+1})$
is the projection over piecewise constant function, namely
\begin{equation}
\begin{aligned}
\label{eq:SAR-3}
 \mathbf{W}_{j}^{n+1}&=\dfrac{1}{\Delta x}\int_{\frac{-\Delta
     x}{2}}^0\tilde{\textbf{W}}(x/\Delta
   t;\mathbf{W}_{j-1}^n,\mathbf{W}_{j}^n)\mathrm d x+\dfrac{1}{\Delta
   x}\int_0^{\frac{\Delta x}{2}}\tilde{\textbf{W}}(x/\Delta
   t;\mathbf{W}_{j}^n,\mathbf{W}_{j+1}^n)\mathrm d x\\
 &=\mathbf{W}_j^n-\dfrac{\Delta t}{\Delta
   x}\left(\lambda_{r,j-1/2}(\mathbf{W}_j^n-\mathbf{W}_{j-1/2}^{*})-\lambda_{\ell,
     j+1/2}(\mathbf{W}_j^n-\mathbf{W}_{j+1/2}^{*})\right),
 \end{aligned}
\end{equation}
where the state $\mathbf{W}_{j+
  1/2}^{*}=(\mathbf{W}_{j+1/2}^{(1),*},\mathbf{W}_{j+1/2}^{(2),*}$)
denotes the intermediate state for the
approximate Riemann solver $\tilde{\mathbf W}\left( x/\Delta t;
  \mathbf W_j^n, \mathbf W_{j+1}^n\right)$, for $j\in \mathbb Z$.

Then, in the presence of a source term, the scheme is expressed in
terms of $\{\mathbf{Q}\}$. Taking the limit \(\varepsilon \to 0\) in \eqref{eq:SAR-3} therefore gives
\begin{equation}
  \label{eps_0}
\begin{aligned}
\mathbf{W}_{j}^{(2),n+1} &= \mathbf{W}_j^{(2),n} - \dfrac{\Delta t}{\Delta x} \Bigg[
\lambda_{r,j-1/2} \mathbf{W}_j^{(2),n}
+ \dfrac{\lambda_{r,j-1/2} \Delta x}{\Delta t [\lambda]_{j-1/2}}
\left(\dfrac{\mathbf{W}_j^{(2),n} +
    \mathbf{W}_{j-1}^{(2),n}}{2}\right)\\
&- \dfrac{\lambda_{r,j-1/2} \Delta x}{\Delta t [\lambda]_{j-1/2}}
\{\mathbf{Q}\}_{j-1,j}+
 \lambda_{r,j-1/2} \dfrac{\lambda_{\ell,j-1/2} \mathbf{W}_{j-1}^{(2),n} -
   \lambda_{r,j-1/2} \mathbf{W}_j^{(2),n}}{[\lambda]_{j-1/2}}\\
 &- \lambda_{\ell,j+1/2} \mathbf{W}_j^{(2),n} 
- \dfrac{\lambda_{\ell,j+1/2} \Delta x}{\Delta t [\lambda]_{j+1/2}}
\left(\dfrac{\mathbf{W}_j^{(2),n} +
    \mathbf{W}_{j+1}^{(2),n}}{2}\right) \\
&\quad + \dfrac{\lambda_{\ell,j+1/2} \Delta x}{\Delta t [\lambda]_{j+1/2}} \{\mathbf{Q}\}_{j,j+1}
- \lambda_{\ell,j+1/2} \dfrac{\lambda_{\ell,j+1/2}
  \mathbf{W}_j^{(2),n} - \lambda_{r,j+1/2}
  \mathbf{W}_{j+1}^{(2),n}}{[\lambda]_{j+1/2}}
\Bigg].
\end{aligned}
\end{equation}
where $[\lambda]_{j\pm1/2}=\lambda_{r,j\pm1/2} - \lambda_{\ell,j\pm1/2}$.
On the other hand, when $\varepsilon \to 0$, we expect
$\mathbf{R}(\mathbf{W}_j^{n+1}) = 0$, which implies that
$\mathbf{W}_j^{(2),n+1} = \mathbf{Q}(\mathbf{W}_j^{(1),n+1})$.
Assuming $\{\mathbf{Q}\}_{j-1,j}=\{\mathbf{Q}\}_{j,j+1}$
and  substituting $\mathbf{W}_j^{(2),n+1}$ with
$\mathbf{Q}(\mathbf{W}_j^{(1),n+1})$
in \eqref{eps_0}, we obtain the formulation for $\{\mathbf{Q}\}$:
\begin{equation}\label{eq:Qapprox}
 \begin{aligned}
   \{\mathbf{Q}\}&=\Bigg[\frac{1}{\lambda_{\ell,j+1/2}[\lambda]_{j-1/2}-\lambda_{r,j-1/2}[\lambda]_{j+1/2}}\Bigg]\\
   &\times\Bigg[[\lambda]_{j-1/2}[\lambda]_{j+1/2}\big(-\mathbf{Q}(\mathbf{W}^{(1),n+1}_j)+\mathbf{W}_j^{(2),n}\big)\\
 &-\frac{\Delta t}{\Delta
   x}\bigg(\lambda_{r,j-1/2}\lambda_{\ell,j-1/2}[\lambda]_{j+1/2}(\mathbf{W}^{(2)}_{j-1}-\mathbf{W}_j^{(2),n})\\
 &+\lambda_{\ell,j+1/2}\lambda_{r,j+1/2}[\lambda]_{j-1/2}(\mathbf{W}_{j+1}^{(2),n}-\mathbf{W}_j^{(2),n})\bigg)\\
 &-\lambda_{r,j-1/2}[\lambda]_{j+1/2}\bigg(\frac{\mathbf{W}^{(2),n}_{j-1}+\mathbf{W}_j^{(2),n}}{2}\bigg)+
 \lambda_{\ell,j+1/2}[\lambda]_{j-1/2}\bigg(\frac{\mathbf{W}_{j+1}^{(2),n}+\mathbf{W}_j^{(2),n}}{2}\bigg)\Bigg].
 \end{aligned}   
\end{equation}
If $\lambda_{r,j-1/2}=\lambda_{r,j+1/2}=\lambda_r$ and
$\lambda_{\ell,j-1/2}=\lambda_{\ell,j+1/2}=\lambda_\ell$,
then it yields
\begin{equation}
  \label{eq:Qapproxsimplify}
  \begin{aligned}
    \{\mathbf{Q}\} &= \mathbf{Q}(\mathbf{W}_j^{(1),n+1}) -
    \mathbf{W}_j^{(2),n}\\
    &\quad + \dfrac{\lambda_r }{\lambda_r -\lambda_\ell}\left(
      \dfrac{\mathbf{W}_{j-1}^{(2),n}+ \mathbf{W}_j^{(2),n}}{2}\right)
    - \dfrac{\lambda_\ell }{\lambda_r -\lambda_\ell}\left(
      \dfrac{\mathbf{W}_{j+1}^{(2),n}+
        \mathbf{W}_j^{(2),n}}{2}\right)\\
       &\quad +\dfrac{\Delta t}{\Delta
         x}\dfrac{\lambda_r\lambda_\ell}{\lambda_r-\lambda_\ell}(\mathbf{W}_{j-1}^{(2),n}-2\mathbf
      {W}_j^{(2),n}+ \mathbf{W}_{j+1}^{(2),n}).
  \end{aligned}
\end{equation}
It follows that the update state simplifies to 
\begin{equation}\label{eq:W1np1}
  \mathbf{W}_j^{(1),n+1}=\mathbf{W}_j^{(1),n}-
  \frac{\Delta t}{\Delta x}\bigg(\mathbf{F}_{j+1/2}^{(1)}-\mathbf{F}_{j-1/2}^{(1)}\bigg),
\end{equation}
with the numerical flux
\begin{equation}
\label{eq:F1}
\begin{aligned}
  \mathbf{F}_{j+1/2}^{(1)}&=\frac{\lambda_r\lambda_\ell}{\lambda_r-\lambda_\ell}
  \bigg(\mathbf{W}_{j+1}^{(1),n}-\mathbf{W}_{j}^{(1),n}\bigg)\\
&-\frac{1}{\lambda_r-\lambda_\ell}\bigg(\lambda_\ell\mathbf{f}^{(1)}(\mathbf{W}_{j+1/2}^R(\Delta
t))-\lambda_r\mathbf{f}^{(1)}(\mathbf{W}_{j+1/2}^{L}(\Delta
t)\bigg),
\end{aligned}
\end{equation}
with
\begin{equation}
    \label{eq:ST-int-i}
    \begin{cases}
        \dfrac{\mathrm d}{\mathrm d t}\mathbf{W}_{j+1/2}^R(t) = \dfrac{1}{\varepsilon}\mathbf R( \mathbf{W}_{j+1/2}^R(t)),\\
        \mathbf{W}_{j+1/2}^R(0) = \mathbf{W}_{j+1}^n,
    \end{cases}
\quad
    \begin{cases}
        \dfrac{\mathrm d}{\mathrm d t}\mathbf{W}_{j+1/2}^L(t) =
        \dfrac{1}{\varepsilon}\mathbf R( \mathbf{W}_{j+1/2}^L(t)),  \text{ for } t>0,\\
        \mathbf{W}_{j+1/2}^L(0) = \mathbf{W}_{j}^n,
    \end{cases}
\end{equation}
and
\begin{equation}
  \label{eq:W2np1}
  \mathbf{W}_j^{(2),n+1}=\mathbf{W}_j^{(2),n}
  -\frac{\Delta t}{\Delta
    x}\bigg(\mathbf{F}_{j+1/2}^{(2)}-\mathbf{F}_{j-1/2}^{(2)}\bigg)
  -\bigg(e^{-\Delta
    t/\varepsilon}-1\bigg)\bigg(\mathbf{Q}(\mathbf{W}_{j}^{(1),n+1})
  -\mathbf{W}_{j}^{(2),n}\bigg),
\end{equation}
with the numerical flux
\begin{equation}
\label{eq:F2}
\begin{aligned}
\mathbf{F}_{j+1/2}^{(2),n}&=\frac{e^{\frac{-\Delta
      t}{\varepsilon}}\lambda_r\lambda_\ell}{\lambda_r-\lambda_\ell}
\bigg(\mathbf{W}_{j+1}^{(2),n}-\mathbf{W}_{j}^{(2),n}\bigg)\\
&+\frac{\varepsilon(e^{-\Delta t/\varepsilon}-1)}{\Delta
  t(\lambda_r-\lambda_\ell)}
\bigg(\lambda_\ell\mathbf{f}^{(2)}(\mathbf{W}^R_{j+1/2}(\Delta t))-
\lambda_r\mathbf{f}^{(2)}(\mathbf{W}_{j+1/2}^L(\Delta t))\bigg).  
\end{aligned}
\end{equation}

\subsection{Properties of the Approximate Riemann solver}
\label{sec:prop-ARS}

As a direct consequence of the consistency
condition \eqref{eq:ApproxConsistencyCondidiotn},
the ARS is consistent with solutions of \eqref{eq:HSR2}.
Moreover, the associated Godunov scheme endowed with the asymptotic
correction
guarantees the scheme to be asymptotic preserving by construction.
When \(\varepsilon=0\), the scheme reduces to the HLL scheme applied
to the limit hyperbolic equilibrium model \eqref{eq:HE}. These
properties are summarized  in the following proposition.

\begin{Prop}{(Asymptotic preserving property)}
  Let the constant sequence of cell-averaged values
  $(\mathbf W_j^{(1),n}, \mathbf W_j^{(2),n})$ be known at time $t^n$,
  for $j\in \mathbb Z$.
  Under the CFL condition \eqref{CFL:ARS}, the scheme
\eqref{eq:W1np1}–\eqref{eq:W2np1} is asymptotic preserving, in the
sense that it is consistent with solutions of the hyperbolic model
\eqref{eq:HSR2} for all $\varepsilon>0$ and, in the limit
$\varepsilon\to 0$, it converges to the stable and consistent HLL
scheme \cite{harten1983upstream}
for the limit hyperbolic equilibrium model \eqref{eq:HE}.
\end{Prop}

In the case of the Jin and Xin model, the scheme \eqref{eq:W1np1}-\eqref{eq:F2} reads
\begin{equation}
  \label{eq:wnp1}
  \begin{aligned}
    u_j^{n+1}&=u_{j}^{n}-\frac{\Delta t}{\Delta
      x}\bigg(\frac{-\lambda}{2}(u_{j+1}^n-2u_j^n+u_{j-1}^n)+\frac{1}{2}\bigg((v_{j+1}^n-g(u_{j+1}^n))e^{\frac{-\Delta
        t}{\varepsilon}}+g(u_{j+1}^n)\\
             &-(v_{j-1}^n-g(u_{j-1}^n))e^{\frac{-\Delta
                 t}{\varepsilon}}-g(u_{j-1}^n)\bigg)\bigg),\\
    v_j^{n+1}&=v_j^n-\frac{\Delta t}{\Delta x}\bigg(\frac{-\lambda
      e^{\frac{-\Delta
          t}{\varepsilon}}}{2}(v_{j+1}^n-2v_j^n+v_{j-1}^n)-\frac{\lambda^2\varepsilon(e^{\frac{-\Delta
          t}{\varepsilon}}-1)}{2\Delta t}(u_{j+1}^n-u_{j-1}^n)\bigg)\\
             &\qquad-(e^{\frac{-\Delta t}{\varepsilon}}-1)(g(u_j^{n+1})-v_j^n).
           \end{aligned}
\end{equation}
In the limit $\varepsilon\to 0$, the scheme reads
\begin{equation*}
  \begin{aligned}
    u_j^{n+1} &= u_j^n - \dfrac{\Delta t}{\Delta x}\left( -
      \dfrac{\lambda}{2}(u_{j+1}^n -2u_j^n+u_{j-1}^n)+ \dfrac 1 2
      (g(u_{j+1}^n)-g(u_{j-1}^n)\right),\\
    v_j^{n+1} &=g(u_j^{n+1}).
  \end{aligned}
\end{equation*}
It corresponds to the Lax-Friedrichs or Rusanov scheme applied to the
equilibrium equation \eqref{eq:SCL} with the CFL condition
corresponding to the wave speeds of the relaxed system \eqref{eq:JX}.

Moreover, within the framework of the Jin and Xin model, it is possible to prove
that the approximate Riemann solver ensures the invariance of the set of admissible
states $K$ for the equilibrium model \eqref{eq:SCL}, see property
\eqref{it:JX-K} in Section \ref{sec:jin-xin-model}.
\begin{Prop}
  Under the subcharacteristic
  condition \eqref{eq:subcar} and the CFL condition \eqref{CFL:ARS},
  if $(u_j^n, v_j^n) \in D_{K}^{\lambda}$ for all $j \in \mathbb{Z}$,
  then $u_j^{n+1}\in K$ for all $j \in \mathbb{Z}$.
\end{Prop}

\begin{proof}
  The update for $u_j^{n+1}$ can be written as
  \begin{equation*}
    \begin{aligned}
      u_j^{n+1} &= u_{j}^n (1-\dfrac{\lambda \Delta t}{\Delta x}) +
      \dfrac{\lambda \Delta t}{2 \Delta x} (u_{j+\demi}^ R(\Delta t) +
      u_{j-\demi}^ L(\Delta t) )\\
      &-
      \dfrac{ \Delta t}{2 \Delta x} (v_{j+\demi}^ R(\Delta t) -
      v_{j-\demi}^ L(\Delta t)).
    \end{aligned}
  \end{equation*}
  Reorganizing the terms and using the definition of $r_{j\pm
    \demi }^{L,R}(\Delta t)$, it yields
  \begin{equation*}
    \begin{aligned}
       u_j^{n+1} &= u_{j}^n (1-\dfrac{\lambda \Delta t}{\Delta x}) +
       \dfrac{\lambda \Delta t}{2 \Delta x}  (r_{j-\demi}^L (\Delta
       t)+s_{j+\demi}^R (\Delta t)).
    \end{aligned}
  \end{equation*}
  According to \ref{prop:JX-Stag-step2-Linf}, the
  sum of the two last terms belongs to
  $\dfrac 1 2 K_+ + \dfrac 1 2
  K_- =K$.
  The first term belongs to $K$. Hence $u_j^{n+1}$ is a convex
  combination of elements of $K$, which concludes the proof.
\end{proof}

Moreover, returning to the definition of the approximate Riemann solver, we can
establish a local entropy stability property. We begin with a general
setting and consider any  solution of the relaxation system
\eqref{eq:HSR2} satisfying the entropy identity
\begin{equation}
\label{eq:EI2}
\partial_t H(\mathbf W) + \partial_x \mathbf \Psi(\mathbf W)
= \mathbf D(\mathbf W),
\end{equation}
where \(H\) is a convex entropy associated with the entropy flux \(\mathbf \Psi\),
and \(\mathbf D(\mathbf W)=\frac{1}{\varepsilon}\,\nabla H(\mathbf W)\cdot \mathbf R(\mathbf W)\)
is the dissipative source contribution. Following
\cite{harten1983upstream,BerthonChalons16}, if we assume that

\begin{equation}
  \label{eq:ARS-inegEntrop}
  \begin{aligned}
    \dfrac{1}{\Delta x} \int_{x_j}^{x_{j+1}}H&\big(\tilde{\mathbf{W}}
    \big(\dfrac{x-x_{j+\demi}}{\Delta t}\big. \big.;
    \big. \big.\mathbf W_j^n, \mathbf
    W_{j+1}^n\big)\big)\mathrm d x\\
    &\leq \dfrac{1}{2} \left(
      H(\mathbf W_j^n)+H(\mathbf W_{j+1}^n)\right) - \dfrac{\Delta
      t}{\Delta x}(\mathbf \Psi(\mathbf W_{j+1}^n)- \mathbf
    \Psi(\mathbf W_{j}^n))\\
    &+ \Delta t \mathbf D(\Delta t, \Delta x ; \mathbf
    W_j^n, \mathbf W_{j+1}^n),
  \end{aligned}
\end{equation}
  for all $j\in \mathbb Z$, 
with $\displaystyle\lim_{\substack{\mathbf W_\ell, \mathbf W_r \to \mathbf
    W\\ \Delta t, \Delta x \to 0}}\mathbf D(\Delta t, \Delta x ; \mathbf
W_\ell, \mathbf W_r) = \mathbf D(\mathbf W)$,
then the numerical scheme \eqref{eq:SAR-3}  satisfies
\begin{equation}
  \label{eq:ARS-inegEntrop-LOC}
  H(\mathbf W_{j}^{n+1}) \leq   H(\mathbf W_{j}^{n})- \dfrac{\Delta
    t}{\Delta x} \big( \mathbf \Psi_{j+\demi} - \mathbf \Psi_{j-\demi}\big) + \Delta t
  \mathbf D_j^n,
\end{equation}
with
\begin{equation}
  \label{eq:Din}
   \mathbf D_j^n = \dfrac{1}{2}\big(  \mathbf D(\Delta t, \Delta x ; \mathbf
W_j^n, \mathbf W_{j+1}^n ) + \mathbf D(\Delta t, \Delta x ; \mathbf
W_{j-1}^n, \mathbf W_{j}^n \big)\big),
\end{equation}
and
\begin{equation}
  \label{eq:Psidemi}
  \begin{aligned}
    \mathbf \Psi_{j+\demi} &= \dfrac{1}{2}\big(\mathbf \Psi (\mathbf W_{j+1}^n)
    - \mathbf \Psi (\mathbf W_{j}^n) \big)
    -\dfrac 1 4 \dfrac{\Delta x}{\Delta t} \big(H(\mathbf W_{j+1}^n)-
    H(\mathbf W_{j}^n)\big)\\
    &+ \dfrac{1}{2\Delta t}\int_{x_j}^{x_{j+\demi}}H\big(\tilde{\mathbf{W}}
        \big(\dfrac{x-x_{j+\demi}}{\Delta t}\big. \big.;
      \big. \big.\mathbf W_j^n, \mathbf
      W_{j+1}^n\big)\big)\mathrm d x\\
      &- \dfrac{1}{2\Delta t}\int_{x_{j+\demi}}^{x_{j+1}}H\big(\tilde{\mathbf{W}}
        \big(\dfrac{x-x_{j+\demi}}{\Delta t}\big. \big.;
      \big. \big.\mathbf W_j^n, \mathbf
      W_{j+1}^n\big)\big)\mathrm d x.
  \end{aligned}
\end{equation}
Indeed, 
since the entropy function $H$  is a
convex, the Jensen inequality gives
\begin{equation*}
  \begin{aligned}
    H(\mathbf W_{j}^{n+1}) &\leq \dfrac{1}{\Delta x}
    \left[
      \int_{x_{j-\demi}}^{x_{j}}H\big(\tilde{\mathbf{W}}
        \big(\dfrac{x-x_{j-\demi}}{\Delta t};
      \mathbf W_{j-1}^n, \mathbf
      W_{j}^n\big)\big)\mathrm d x \right.\\
      &\quad +
      \left.\int_{x_{j}}^{x_{j+\demi}}H\big(\tilde{\mathbf{W}}
        \big(\dfrac{x-x_{j+\demi}}{\Delta t};
      \mathbf W_{j}^n, \mathbf
      W_{j+1}^n\big)\big)\mathrm d x
    \right]\\
    &\leq  \dfrac{1}{2\Delta x} \int_{x_{j-\demi}}^{x_{j}}H\big(\tilde{\mathbf{W}}
        \big(\dfrac{x-x_{j-\demi}}{\Delta t};
      \mathbf W_{j-1}^n, \mathbf
      W_{j}^n\big)\big)\mathrm d x\\
      &+ \dfrac{1}{2\Delta x} \int_{x_{j-1}}^{x_{j}}H\big(\tilde{\mathbf{W}}
        \big(\dfrac{x-x_{j-\demi}}{\Delta t};
      \mathbf W_{j-1}^n, \mathbf
      W_{j}^n\big)\big)\mathrm d x \\
      &- \dfrac{1}{2\Delta x} \int_{x_{j-1}}^{x_{j-\demi}}H\big(\tilde{\mathbf{W}}
        \big(\dfrac{x-x_{j-\demi}}{\Delta t};
      \mathbf W_{j-1}^n, \mathbf
      W_{j}^n\big)\big)\mathrm d x \\
      & +\dfrac{1}{2\Delta x} \int_{x_{j}}^{x_{j+\demi}}H\big(\tilde{\mathbf{W}}
        \big(\dfrac{x-x_{j+\demi}}{\Delta t};
      \mathbf W_{j}^n, \mathbf
      W_{j+1}^n\big)\big)\mathrm d x\\
      &+ \dfrac{1}{2\Delta x} \int_{x_{j}}^{x_{j+1}}H\big(\tilde{\mathbf{W}}
        \big(\dfrac{x-x_{j+\demi}}{\Delta t};
      \mathbf W_{j}^n, \mathbf
      W_{j+1}^n\big)\big)\mathrm d x \\
      &- \dfrac{1}{2\Delta x} \int_{x_{j+\demi}}^{x_{j+1}}H\big(\tilde{\mathbf{W}}
        \big(\dfrac{x-x_{j+\demi}}{\Delta t};
      \mathbf W_{j}^n, \mathbf
      W_{j+1}^n\big)\big)\mathrm d x.
    \end{aligned}
\end{equation*}
If the approximate Riemann solver satisfies
\eqref{eq:ARS-inegEntrop}, then the local entropy inequality
\eqref{eq:ARS-inegEntrop-LOC} holds with the \textit{ad hoc}
definitions of the numerical entropy fluxes and source term
approximation \eqref{eq:Din}-\eqref{eq:Psidemi}.

Actually, the left-hand side of \eqref{eq:ARS-inegEntrop} can be written explicitly and reads
\begin{equation}
  \label{SRA-entropy0}
  \begin{aligned}
    \dfrac{1}{\Delta x} \int_{x_j}^{x_{j+1}}H&\big(\tilde{\mathbf{W}}
        \big(\dfrac{x-x_{j+\demi}}{\Delta t}\big. \big.;
      \big. \big.\mathbf W_j^n, \mathbf
      W_{j+1}^n\big)\big)\mathrm d x =\\
     & \frac{1}{2} \left[
      H(\mathbf W_j^n) + H(\mathbf W_{j+1}^n) \right] + \frac{\Delta
      t}{\Delta x} (
    \lambda_{j+\demi,r}-\lambda_{j+\demi,\ell})H(\mathbf
    W_{j+\demi}^*)\\
    &+ \frac{\Delta
      t}{\Delta x} (\lambda_{j+\demi,\ell} H(\mathbf W_{j}^n)- \lambda_{j+\demi,r} H(\mathbf W_{j+1}^n)). 
  \end{aligned}
\end{equation}
Following \cite{Martaud25}, in order to prove the validity of a local discrete entropy inequality,
it is then sufficient to prove that
\begin{equation}
  \label{SRA-entropy}
  \begin{aligned}
   H(\mathbf W_{j+\demi}^*)&\leq\dfrac{1}{
     \lambda_{j+\demi,r}-\lambda_{j+\demi,\ell}}\left(
     \lambda_{j+\demi,r} H(\mathbf W_{j+1}^n)
    -
       \lambda_{j+\demi,\ell} H(\mathbf W_{j}^n))\right)\\
    & - \dfrac{1}{
       \lambda_{j+\demi,r}-\lambda_{j+\demi,\ell}}
     (\mathbf{\Psi}(\mathbf W_{j+1}^n) -\mathbf{\Psi}(\mathbf W_{j}^n))
     + \Delta t D(\Delta t, \Delta x; \mathbf W_{j}^n, \mathbf W_{j+1}^n).
  \end{aligned}
\end{equation}
Since the numerical fluxes depend on the exact solution of the source–term ODEs and the asymptotic-preserving correction involves implicit terms, one can establish \eqref{SRA-entropy} up to a remainder of the form \(\Delta t\,\mu(\Delta t)\), where \(\mu(\Delta t)\to0\) as \(\Delta t\to0\).
The calculations are detailed for the Jin–Xin model—namely, for the numerical scheme \eqref{eq:wnp1}—along the lines of \cite{BerthonChalons16}.

First, we rewrite the scheme \eqref{eq:wnp1} as a perturbation of the HLL scheme \cite{harten1983upstream} applied to the homogeneous system \eqref{eq:HSR2} (i.e., with the source term suppressed).
Using an asymptotic expansion in the small parameter \(\Delta t/\varepsilon\) near zero, we obtain

\begin{equation}
  \label{eq:JX-HLL-corr}
  \begin{aligned}
    u_j^{n+1} = u_j^{\text{HLL}}+ \Delta t
    \mu_u(\Delta t),\qquad
    v_j^{n+1} = v_j^{\text{HLL}}+ \Delta t
    \mu_v(\Delta t)+ \Delta t q_j
  \end{aligned}
\end{equation}
with
\begin{equation}
  \label{eq:JX-HLL}
  \begin{aligned}
    u_j^{\text{HLL}} &= u_j^n -\dfrac{\Delta t}{\Delta x} \left(
      -\dfrac{\lambda}{2}(u_{j+1}^n-2u_j^n+u_{j-1}^n) +
      \dfrac{1}{2}(v_{j+1}^n -v_{j-1}^n)\right),\\
    v_j^{\text{HLL}}&=v_j^n - \dfrac{\Delta t}{\Delta x} \left(
      -\dfrac{\lambda}{2}(v_{j+1}^n-2v_j^n+v_{j-1}^n) +
      \dfrac{\lambda^2}{2}(u_{j+1}^n -u_{j-1}^n)\right)
  \end{aligned}
\end{equation}
and
\begin{equation}
  \label{eq:JX-HLL-corr-mu}
  \begin{aligned}
   \mu_u(\Delta t) &= \dfrac{\Delta t }{2\varepsilon \Delta x}
   (v_{j+1}^n- g(u_{j+1}^n) -v_{j-1}^n+g(u_{j-1}^n)) + \mu(\Delta
   t/\varepsilon ),\\
   \mu_v(\Delta t) &= -\dfrac{\lambda \Delta t}{2\varepsilon \Delta x}
   (v_{j+1}^n -2v_j^n+v_{j-1}^n) + \dfrac{\lambda^2 \Delta
     t}{2 \varepsilon \Delta x}(u_{j+1}^n -u_{j-1}^n) + \mu(\Delta
   t/\varepsilon ),\\
   q_j &= \dfrac{1}{\varepsilon}(g(u_j^{n+1})-v_j^n)).
  \end{aligned}
\end{equation}
Here, $\mu(\Delta t)$ denotes a function satisfying $\lim_{\Delta t\to 0}\mu(\Delta t)=0$.
Using the expression of $u_j^{n+1}$, the term $q_j$ reads
\begin{equation}
  \label{eq:D_j}
     q_j = \dfrac{1}{\varepsilon}(g(u_j^{HLL})-v_j^n))+ \dfrac 1
     \varepsilon g'(u_j^{HLL})\Delta t \mu_u(\Delta t),
   \end{equation}
   that is to say  $\displaystyle\lim_{\substack{(u_j,v_j)\to \mathbf
    (u,v), \forall j\in \mathbb Z\\ \Delta t, \Delta x \to 0}}
q_j=\dfrac{1}{\varepsilon}(g(u)-v)$.
Now using the expression of the entropy $H$, it holds
\begin{equation}
  \label{eq:JX-H}
  \begin{aligned}
    H(u_j^{n+1},v_j^{n+1}) &=  H(u_j^{HLL},v_j^{HLL})
    + \p_u H(u_j^{HLL},v_j^{HLL})\Delta t \mu_u(\Delta t) \\
    &\quad+ \p_v
    H(u_j^{HLL},v_j^{HLL})(\Delta t\mu_v(\Delta t)+\Delta t q_j).
  \end{aligned}
\end{equation}
Since the
HLL scheme is entropy satisfying \cite{harten1983upstream}, it holds
\begin{equation*}
  H(\mathbf W_j^{\text{HLL}}):= H(u_j^{HLL},v_j^{HLL}) \leq H(\mathbf W_j^n) -
  \dfrac{\Delta t}{\Delta x} (\Psi_{j+\demi}^{\text{HLL}} - \Psi_{j-\demi}^{\text{HLL}})
\end{equation*}
with an appropriate numerical entropy flux $\Psi_{j\pm\demi}^{\text{HLL}}$.
Combining this with \eqref{eq:JX-H} yields
\begin{equation}
  \label{eq:JX-H2}
  \begin{aligned}
    H(\mathbf W_j^{n+1}) =
    H(\mathbf W_j^n) -
    \dfrac{\Delta t}{\Delta x} (\Psi_{j+\demi}^{\text{HLL}} -
    \Psi_{j-\demi}^{\text{HLL}})
    + \Delta t D_j
  \end{aligned}
\end{equation}
with
    \begin{equation}
  \label{eq:JX-H3}
  \begin{aligned}
   D_j &=\p_u H(u_j^{HLL},v_j^{HLL})\mu_u(\Delta t) + \p_v
    H(u_j^{HLL},v_j^{HLL})(\mu_v(\Delta t)+ q_j).
  \end{aligned}
\end{equation}
Finally the numerical scheme \eqref{eq:wnp1} satisfies a local entropy
inequality in the sense of \eqref{eq:ARS-inegEntrop-LOC}.
\section{Numerical comparison of the two schemes}
\label{sec:numerical-results}

This section presents numerical results obtained with the staggered scheme of Section~\ref{sec:staggered-scheme} and the approximate Riemann solver of Section~\ref{sec:ARS} for the models introduced in Section~\ref{sec:exemples}. The results are compared against a reference solution computed with a splitting scheme on a fine mesh. The latter consists of two substeps:
\begin{enumerate}
\item \textbf{Convective step} (from $t^n$ to $t^{n+\demi}$):
  \[
    \mathbf{W}_j^{\,n+\frac{1}{2}}
    = \mathbf{W}_j^{\,n}
      - \frac{\Delta t}{\Delta x}
        \big(\mathbf{F}_{j+\frac{1}{2}}^{\,n}
             - \mathbf{F}_{j-\frac{1}{2}}^{\,n}\big).
  \]
\item \textbf{Source step} (from $t^{n+\demi}$ to $t^{n+1}$):
  \[
    \mathbf{W}_j^{\,n+1}
    = \mathbf{W}_j^{\,n+\frac{1}{2}}
      + \frac{\Delta t}{\varepsilon}\,
        \mathbf{R}\!\big(\mathbf{W}_j^{\,n+1}\big).
  \]
\end{enumerate}
Here $\mathbf{F}_{j+\frac{1}{2}}^{\,n}$ denotes the HLL numerical flux \cite{harten1983upstream}.
Convergence properties of this splitting scheme for the Jin–Xin model were proved in \cite{LS01} using an entropy method; see also \cite{FR13} for uniform convergence in $\varepsilon$ and $\Delta x$ with detailed error estimates.

In all the numerical tests below, the staggered scheme is run under the CFL condition \eqref{CFLstaggered}, while \eqref{CFL:ARS} is used for the approximate Riemann solver. Homogeneous Neumann boundary conditions are imposed at the domain boundaries.


\subsection{The Jin and Xin model}
\label{sec:JX-numerics}

We consider the Jin–Xin relaxation model \eqref{eq:JX} with
\(g(u)=\tfrac12 u^{2}\) for \(u\in\mathbb R\) and \(\lambda=2\).
As \(\varepsilon\to0\), solutions of \eqref{eq:JX} converge to solutions of
the Burgers’ equation \eqref{eq:SCL}.

\begin{figure}[ht]
  \centering
  \includegraphics[width=0.7 \textwidth]{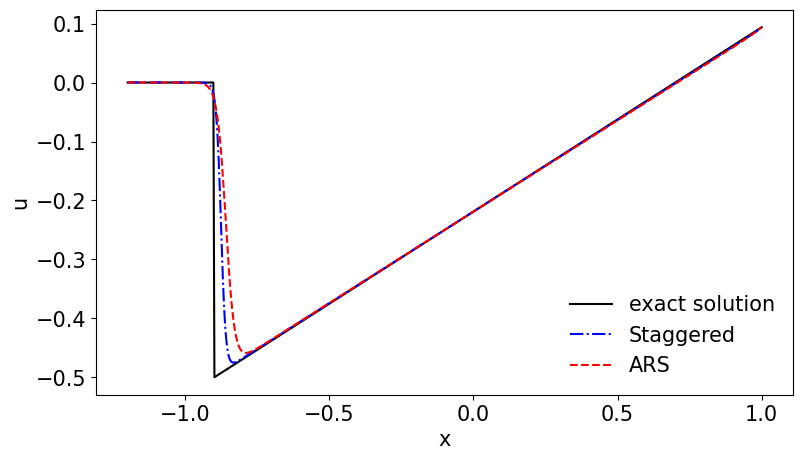}
  \caption{Solutions for the limiting behavior $\varepsilon=10^{-6}$ of
    Jin-Xin model on a 500-cell mesh at $t_{final}=3.2$}
    \label{fig:JXeps_0}
\end{figure}

Figure \ref{fig:JXeps_0} presents the results of a Riemann initial
data problem with
\begin{equation}
  \label{eq:init_cond}
  u_0(x) =
  \begin{cases}
    0, & x < 0.3, \\
    -1, & x \in (0.3,0.7), \\
    \frac{1}{2}, & x > 0.7,
\end{cases}
\end{equation}
and $v_0(x) = g(u_0(x))$, that is the initial data at equilibrium.
The relaxation parameter is set to $\varepsilon=10^{-6}$ such that the
computed profiles can be compared to the equilibrium solution, which
is composed of a shock combined with a rarefaction wave. The
computational domain is made of 500 cells, the CFL parameter is set to
$0.9$ and the resultats are represented at $t_{final} = 3.2$.
One observes the good asymptotic behaviour of both schemes, the approximate
Riemann solver being more diffusive than the staggered scheme.

\begin{figure}[ht]
  \centering
  \includegraphics[width=0.45\textwidth]{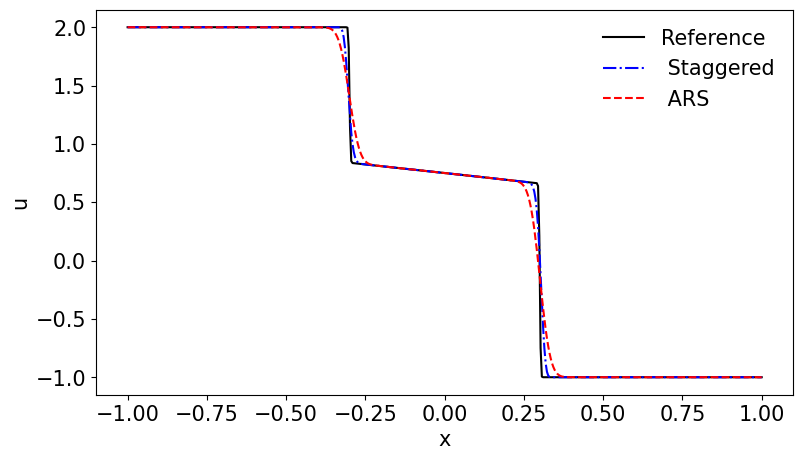}
  \includegraphics[width=0.45\textwidth]{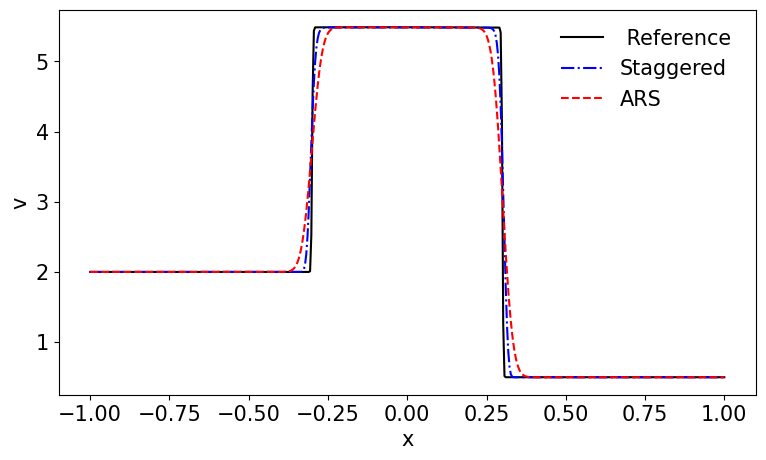}
  
  \includegraphics[width=0.45\textwidth]{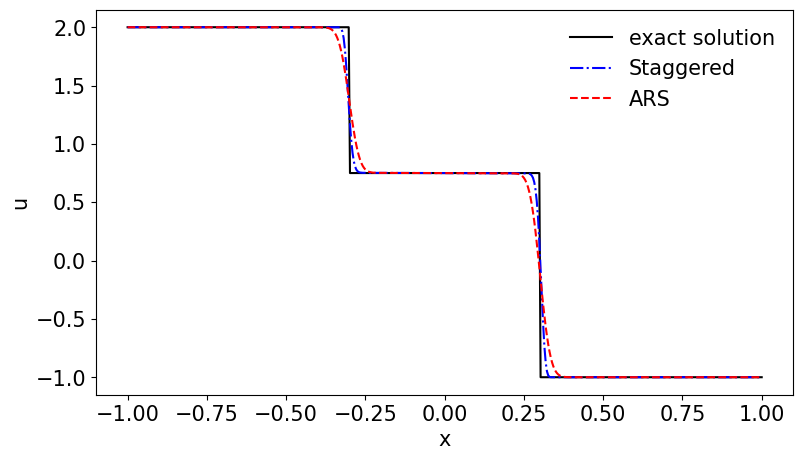}
  \includegraphics[width=0.45\textwidth]{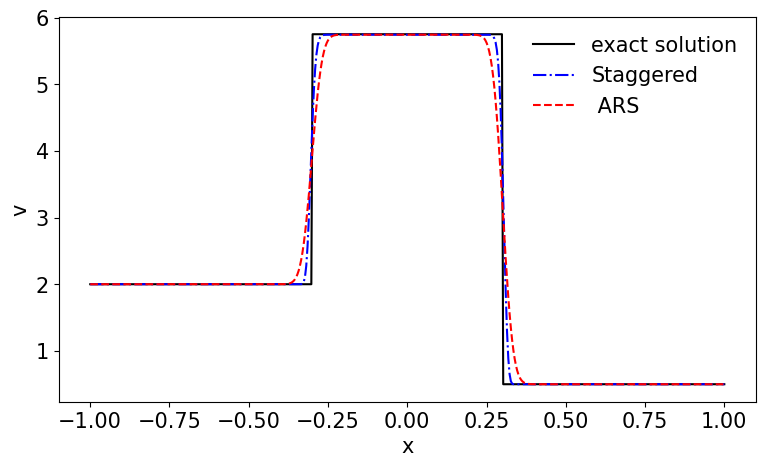}
  
  \caption{Comparison of the solutions obtained with the numerical
    schemes on a 500-cell mesh at final time $T = 0.1$
    for $\varepsilon=1$ (top) and $\varepsilon=40$ (bottom).}
    \label{fig:JXeps}
\end{figure}
Figure~\ref{fig:JXeps} shows solutions of the Jin–Xin model for larger values of \(\varepsilon\).
The reference solution is computed with the HLL splitting scheme on a fine mesh of \(10{,}000\) cells over \((-1,1)\).
The initial data are
\[
(u,v)(0,x)=
\begin{cases}
(2,\,2), & x<0,\\[1mm]
(-1,\,0.5), & x\ge 0,
\end{cases}
\]
and the computational mesh consists of \(500\) cells. Simulations are run up to \(T=0.1\) with CFL \(=0.9\) and \(\lambda=3\).
The top panels of Figure~\ref{fig:JXeps} correspond to \(\varepsilon=1\), while the bottom panels correspond to \(\varepsilon=40\).
Both schemes behave similarly. For large \(\varepsilon\), the solutions develop extended plateaus and approach the hyperbolic solution of the homogeneous Jin–Xin model.

\begin{figure}[ht]
  \centering
  \includegraphics[width=0.7\textwidth]{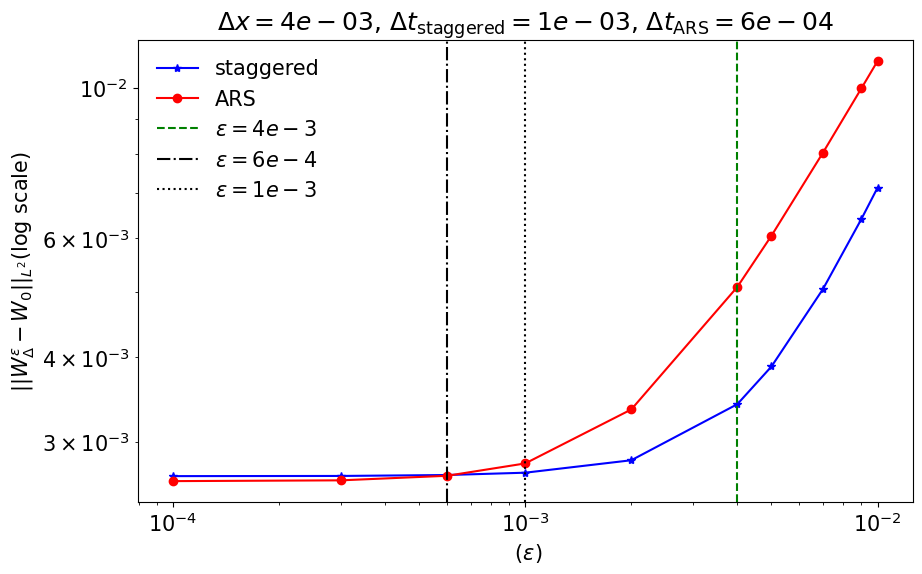}
  
  \includegraphics[width=0.7\textwidth]{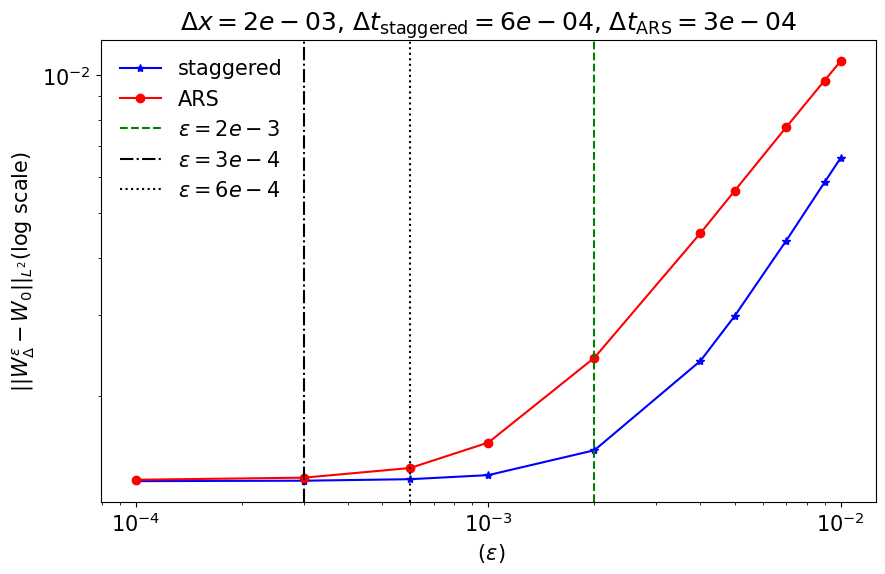} 
  \caption{\(L^2\)-norm error between the numerical solution and an
    exact solution to the Burger's equation for small values of
    $\varepsilon$ at fixed $\Delta x=4 \times 10^{-3}$ (top) and
    $\Delta x=2 \times 10^{-3}$ (bottom).}
  \label{fig:erreurL2_XJ_N500}
\end{figure}

Figure \ref{fig:erreurL2_XJ_N500} illustrates the accuracy of the
schemes with respect to $\varepsilon$. For small value of
$\varepsilon$, the numerical solutions are compared with
the exact solution to the
Burger's equation with initial profile $u_0(t,x) = \frac{x}{1 + t}$. 
The plots report the $L^2$-norm error for a fixed space step
$\Delta x=4 \times 10^{-3}$ (top) and $\Delta x=2 \times 10^{-3}$.

Note that the CFL constraints differ between the two schemes.

In the stiff regime, where $\varepsilon$ is much smaller than both $\Delta x$ and $\Delta t$, 
the two schemes exhibit very low error, confirming their asymptotic
preserving (AP) property.
As $\varepsilon$ increases, different behaviors are observed depending on the scheme. 
For instance, in Figure \eqref{fig:erreurL2_XJ_N500}-bottom,
the error of the approximate Riemann solver, which uses a time step
$\Delta t_{\text{ARS}} = 3\text{e}-4$,
 begins to increase significantly once $\varepsilon$ exceeds $\Delta t_{\text{ARS}}$. 
 This suggests a degradation of the AP behavior when the relaxation parameter becomes 
 larger than the time step. Such behaviour is well-known, see
\cite{jin2010asymptotic}  for instance.
A similar behaviour is observed for the staggered scheme, as
$\varepsilon> \Delta t_{\text{staggered}} = 6\text{e}-4$.
For both schemes, a noticeable loss of accuracy is observed when 
$\varepsilon$ becomes larger than the spatial mesh size $\Delta x = 2\text{e}-3$. 
This behavior reflects the transition out of the equilibrium regime, where relaxation 
no longer dominates, and non-equilibrium effects become significant.

\begin{figure}[ht]
  \centering
  \includegraphics[width=0.7\textwidth]{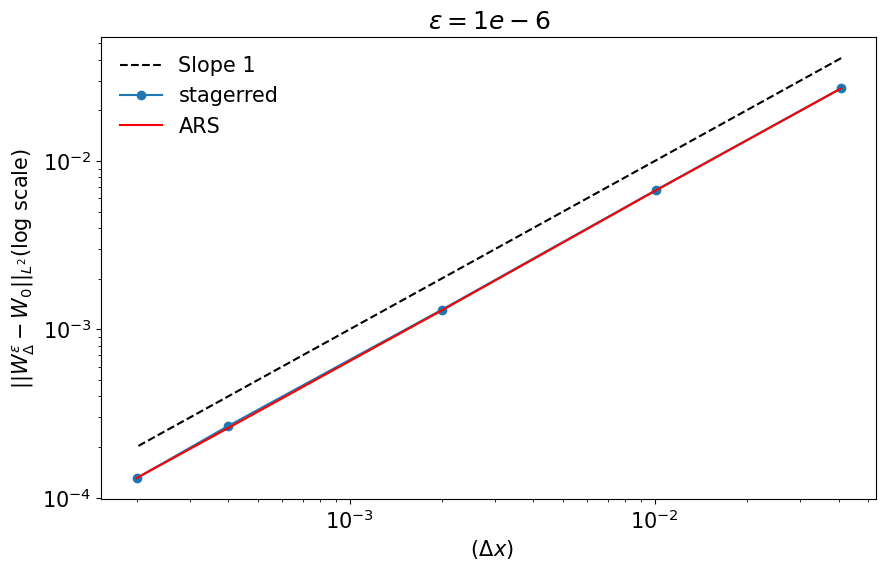} 
  \caption{\(L^2\)-norm error between the numerical solution and the
    exact solution $W_0 $ at fixed relaxation parameter $\varepsilon
    = 10^{-6}  $
    for varying mesh sizes $\Delta x$.}
  \label{fig:convergence_plot_dx}
\end{figure}

To finish, Figure \ref{fig:convergence_plot_dx} illustrates
the first-order convergence in space for both the
staggered and ARS schemes for a fixed $\varepsilon=10^{-6}$ .


\subsection{The Chaplygin model}
\label{sec:Chaplygin-numerics}

We consider the Chaplygin model with the initial data 
$$
(\tau,u,T)(0, x) =
\begin{cases} 
(1,0,1), & \text{if } x < 0, \\ 
(0.8,0,0.8), & \text{otherwise.}
\end{cases}
$$
The computational mesh consists of 1000 cells, with a final simulation
time of $T = 0.1$.
The model parameters are $a=1.8$ and $\gamma=1.4$.

\begin{figure}[ht]
        \centering
        \includegraphics[width=0.7\textwidth]{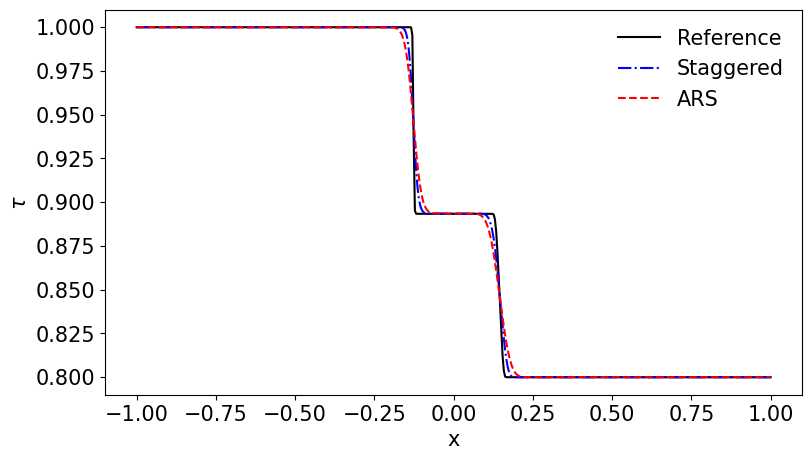}
        \includegraphics[width=0.7\textwidth]{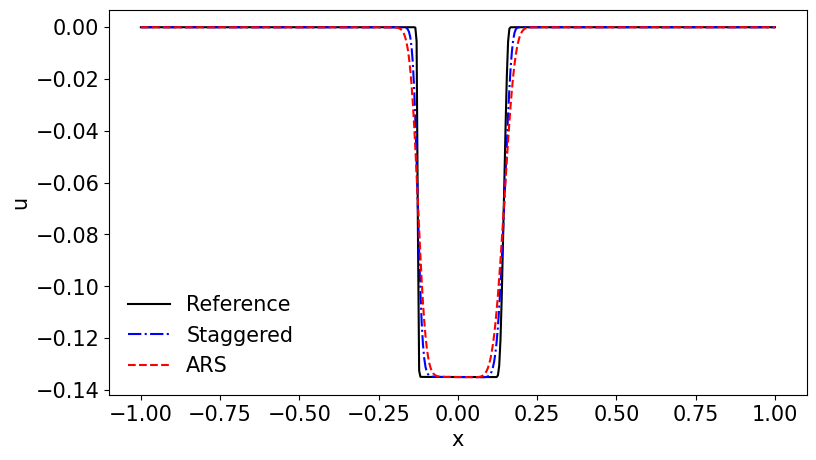}
       \includegraphics[width=0.7\textwidth]{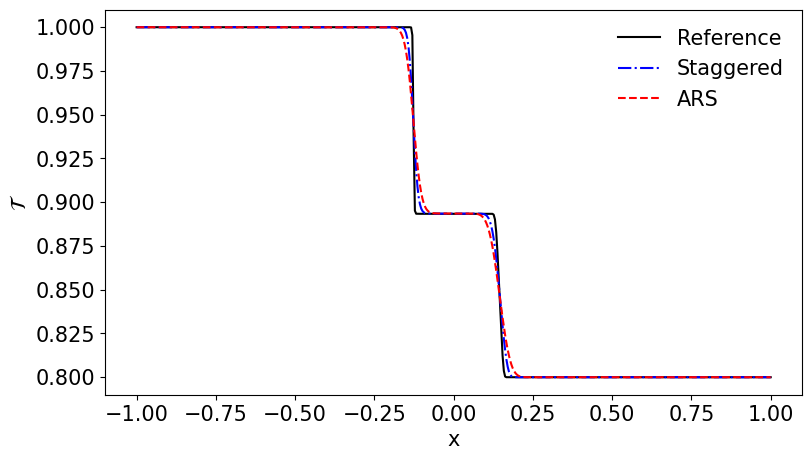}
        \caption{Profiles of covolume $\tau$ (top), velocity $u$
          (middle)
          and relaxed variable $T$ for the Chaplygin model for $\varepsilon = 10^{-6}$.}
  \label{fig:chapeps_0}
\end{figure}

\begin{figure}[ht]
        \centering
        \includegraphics[width=0.7\textwidth]{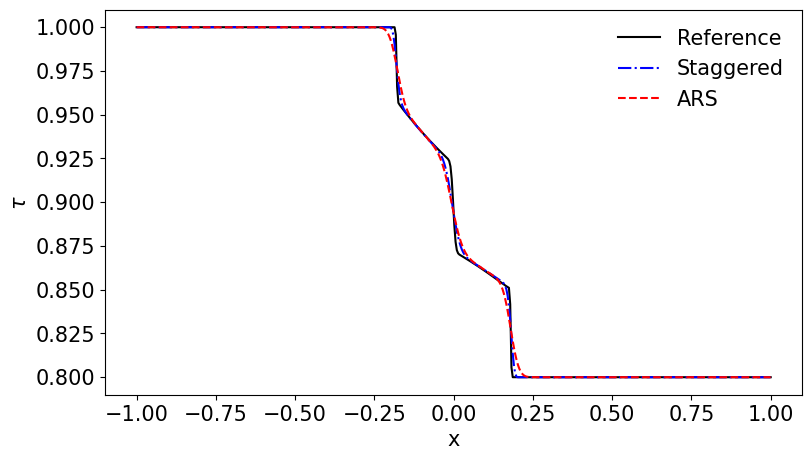}
        \includegraphics[width=0.7\textwidth]{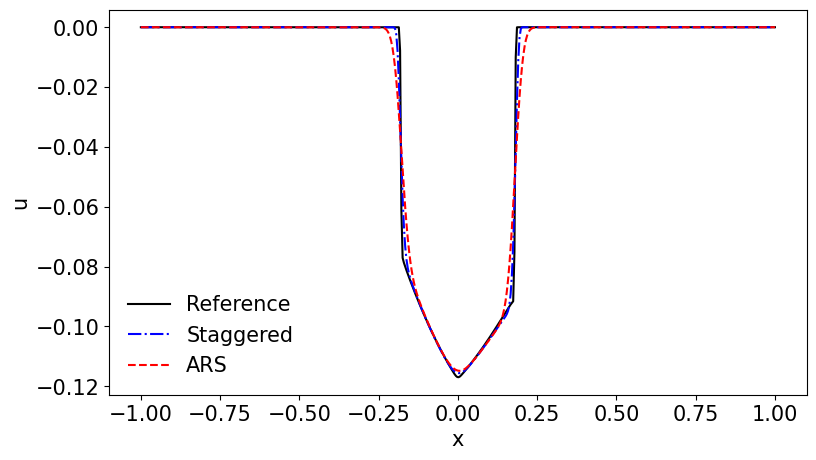}
        \includegraphics[width=0.7\textwidth]{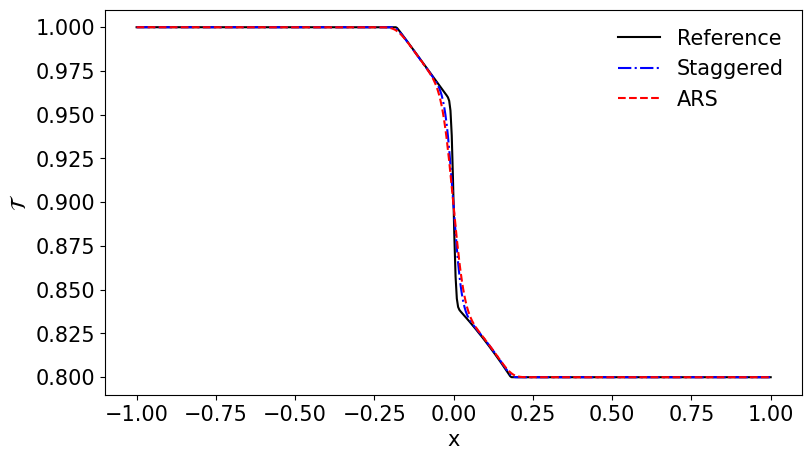}
        \caption{Profiles of covolume $\tau$ (top), velocity $u$
          (middle)
          and relaxed variable $T$ for the Chaplygin model for $\varepsilon = 1$.}
  \label{fig:chapeps_1}
\end{figure}

\begin{figure}[ht]
        \centering
        \includegraphics[width=0.7\textwidth]{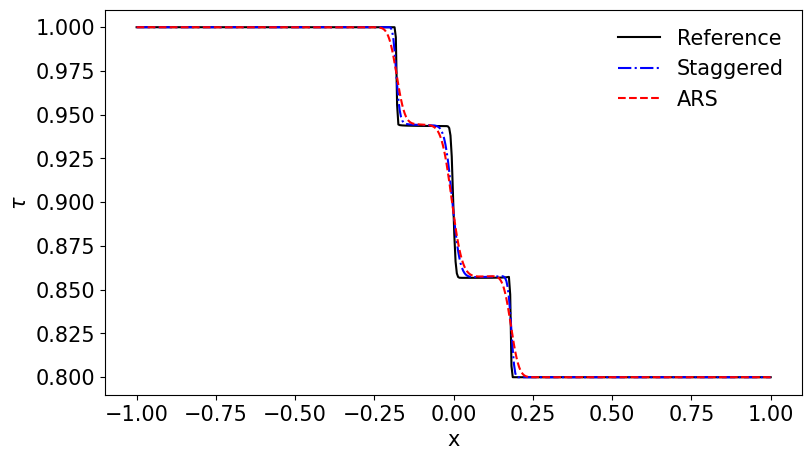}
        \includegraphics[width=0.7\textwidth]{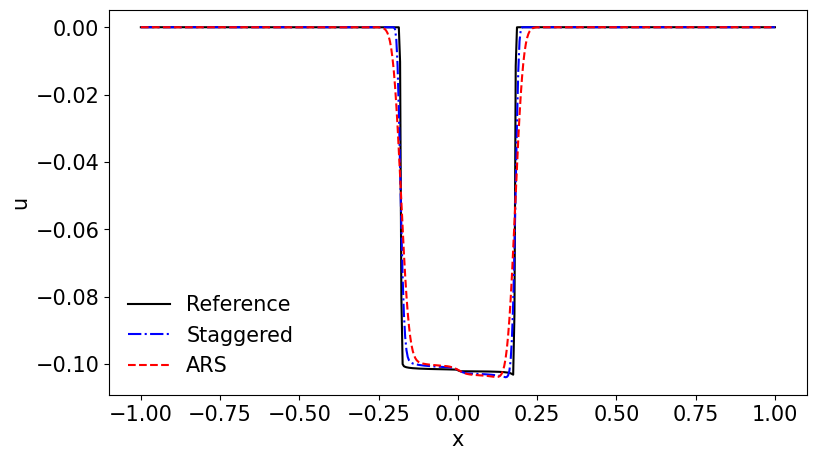}
        \includegraphics[width=0.7\textwidth]{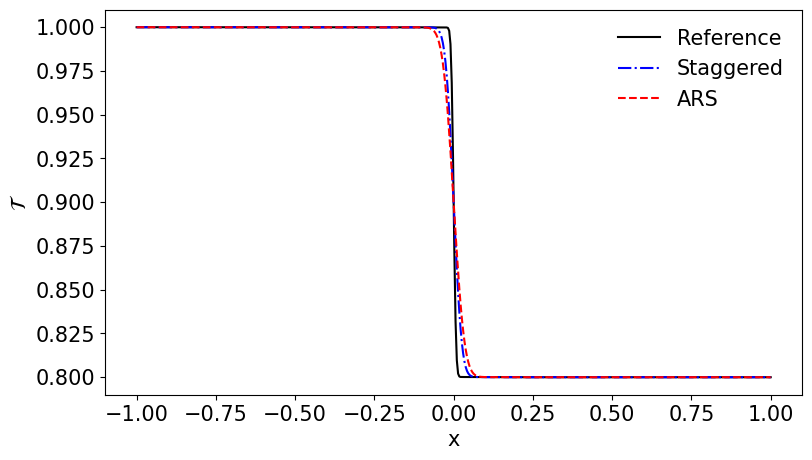}
        \caption{Profiles of covolume $\tau$ (top), velocity $u$
          (middle)
          and relaxed variable $T$ for the Chaplygin model for $\varepsilon = 40$.}
  \label{fig:chapeps_40}
\end{figure}

Figures \ref{fig:chapeps_0}, \ref{fig:chapeps_1}
and \ref{fig:chapeps_40}
present the profile of covolume $\tau$, velocity $u$ and relaxed
variable $T$ for $\varepsilon=0$, $\varepsilon=1$ and $\varepsilon=40$
respectively. For every simulations the reference solution is computed
using a splitting method on a 10,000-cell mesh with $t_{final} = 0.1$.

Again the numerical results illustrate that the proposed schemes are
asymptotic preserving.

\subsection{The two-phase model}
We now consider the compressible two-phase flow model with relaxation
\eqref{eq:phase_transition_HSR}, governed by the perfect gas
coefficients $\gamma_1 = 1.6$ and $\gamma_2 = 1.5$.
The computational domain is $[-0.5,\, 0.5]$, and the final time is set
to $T_{\max} = 0.5$.
The initial data corresponds to a Riemann problem centered at $x = 0$,
with left and right states given by $(\rho_L, u_L, p_L) = (1/0.92,\,
0.4301,\, 0.1445)$ and $(\rho_R, u_R, p_R) = (1/1.3,\, 0.3,\, 0.1)$,
respectively.
Initially the relaxation variable $\varphi$ is set to equilibrium,
namely
$\varphi_{L,R} = \varphi_{eq}(\rho_{L,R})$.
All methods use the CFL condition CFL =
0.9418.
A fine grid with $N = 3000$ is
used for the reference solution obtained by splitting method, while a
coarser grid with $N = 500$ is employed
for the ARS and staggered schemes.

We consider three values of the relaxation parameter: $\varepsilon =
0$, $\varepsilon = 0.1$, and $\varepsilon = 10^4$.
\begin{Rem}
In this case, condition \eqref{condition jacob} is not satisfied; therefore, in the approximate Riemann solver we directly use approximation \eqref{eq:W2}.
\end{Rem}

When $\varepsilon \rightarrow 0$, the system reduces to the
thermodynamic equilibrium model \eqref{eq:phase_transition_HE} with
the equilibrium pressure law $p_{eq}$ given by \eqref{eq:peq}.
Figure \ref{fig:HRM_eps0} show that both schemes match the reference
solution accurately.
The variables $\rho$, $p$, $u$, and $\varphi$ exhibit sharp
transitions that are well resolved. This confirms that both the
Staggered and ARS schemes are asymptotic-preserving in the stiff
regime.

This case $\varepsilon = 0.1$ corresponds to an intermediate
relaxation regime
where the source term is active but not stiff.
Here, $\varphi$ evolves toward equilibrium but has not yet reached it.
As a result, the solution is in a non-equilibrium state. The effect of
the relaxation is visible in the smooth variation of $\varphi$, which
differs from the equilibrium profile. Correspondingly, the density,
pressure and velocity fields also deviate slightly from the
equilibrium structure.
Both schemes remain stable in this regime and give consistent
results. We can observe that, on certain waves, the ARS scheme appears
more diffusive than the staggered one.

In the weak relaxation regime with $\varepsilon = 10^4$,
the source term effect is negligible and the system behaves like an
Euler system coupled with the transport of the mass fraction.
The numerical results show that $\varphi$ remains far from the
equilibrium profile.
This behavior is consistent with the nature of the model in this
regime. The other variables pressure, density, and velocity evolve
according to the conservative convection dynamics.

\label{sec:phase-transition-numerics}
\begin{figure}[ht]
        \centering
        \includegraphics[width=0.7\textwidth]{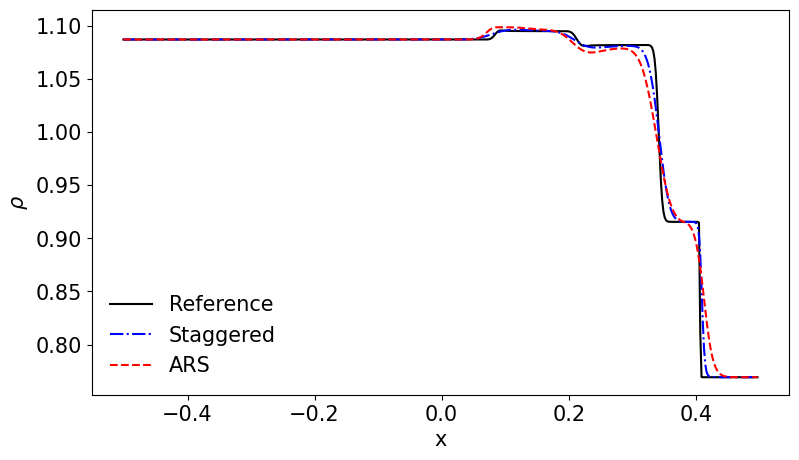}
        \includegraphics[width=0.7\textwidth]{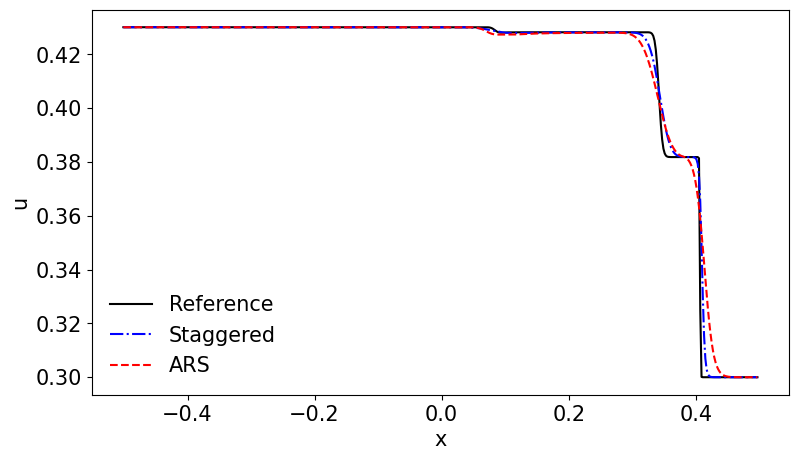}
        \includegraphics[width=0.7\textwidth]{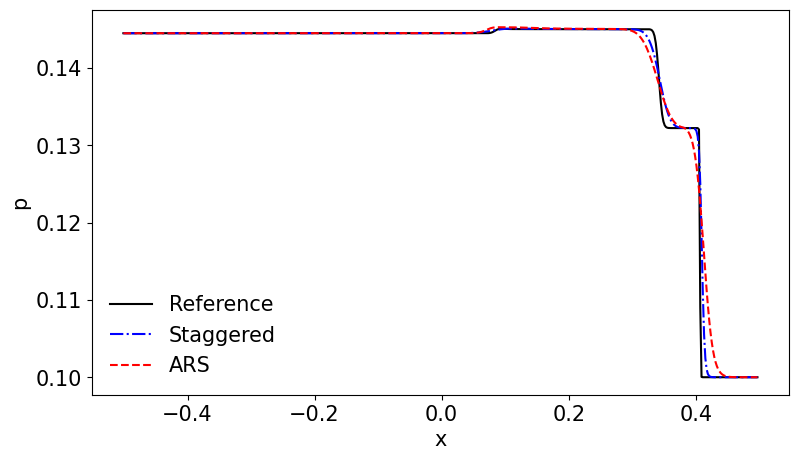}
        \includegraphics[width=0.7\textwidth]{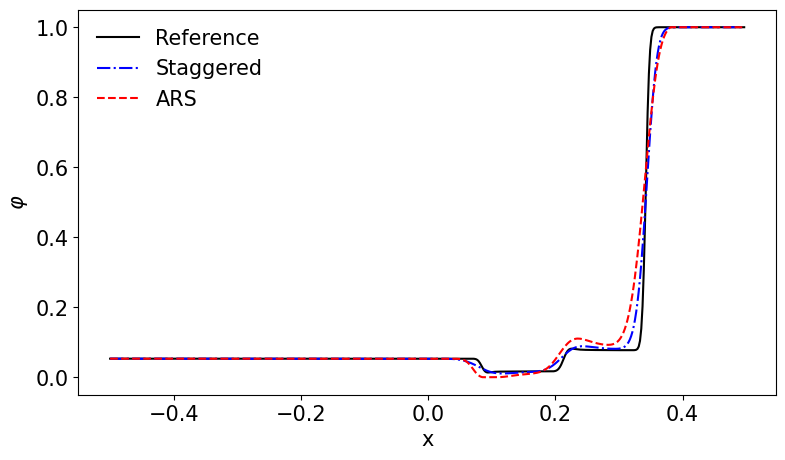}
        \caption{Profiles of density $\rho$ (top), velocity $u$
          (middle), pression $p$
          and relaxed variable $\varphi$ for the HRM model for $\varepsilon = 10^{-15}$.}
  \label{fig:HRM_eps0}
\end{figure}
\begin{figure}[ht]
        \centering
        \includegraphics[width=0.7\textwidth]{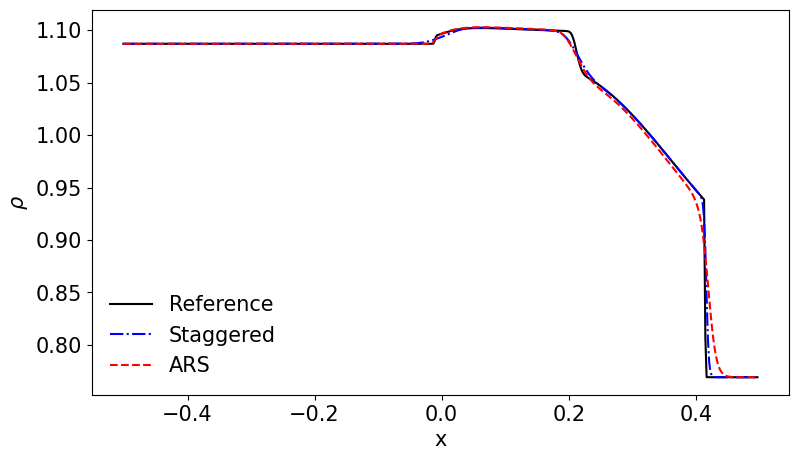}
        \includegraphics[width=0.7\textwidth]{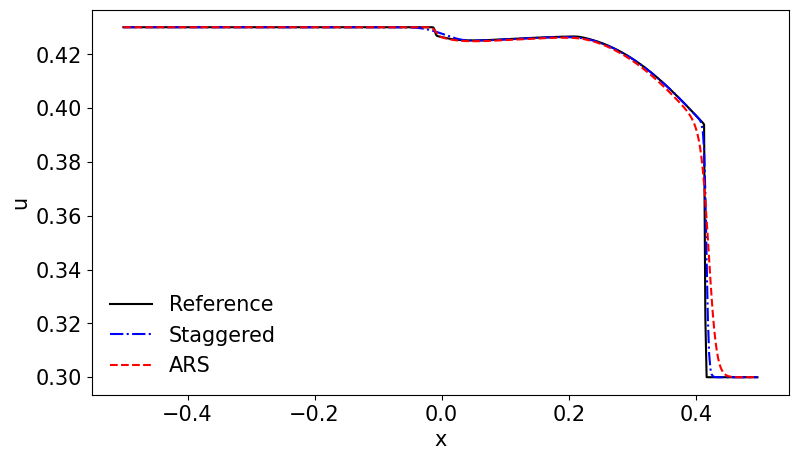}
        \includegraphics[width=0.7\textwidth]{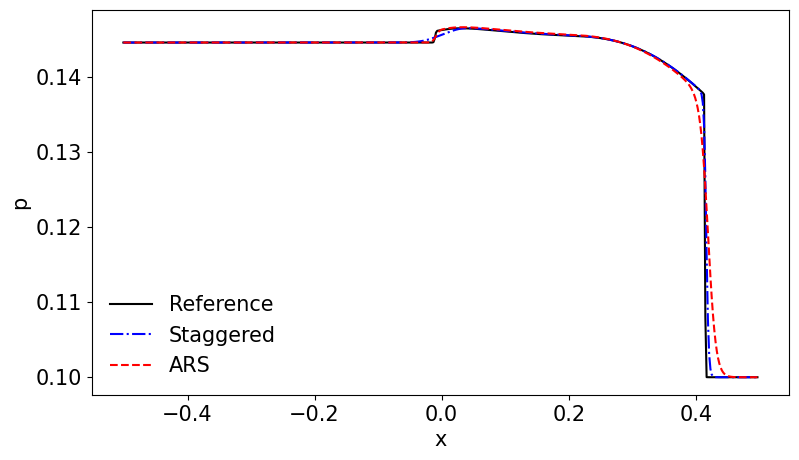}
        \includegraphics[width=0.7\textwidth]{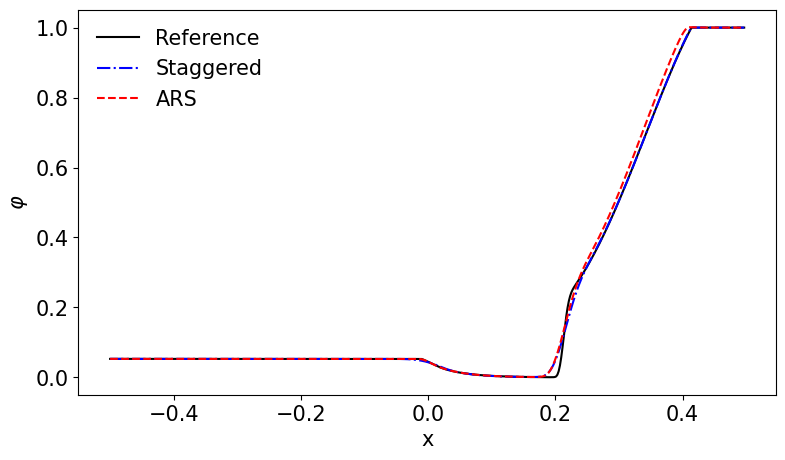}
        \caption{Profiles of density $\rho$ (top), velocity $u$
          (middle), pression $p$
          and relaxed variable $\varphi$ for the HRM model for $\varepsilon = 0.1$.}
  \label{fig:HRM_eps0.1}
\end{figure}
\begin{figure}[ht]
    \centering
    \includegraphics[width=0.7\textwidth]{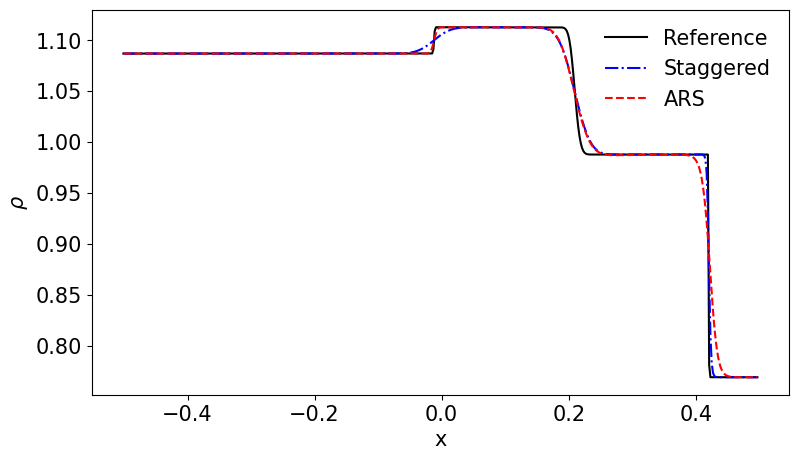}
    \includegraphics[width=0.7\textwidth]{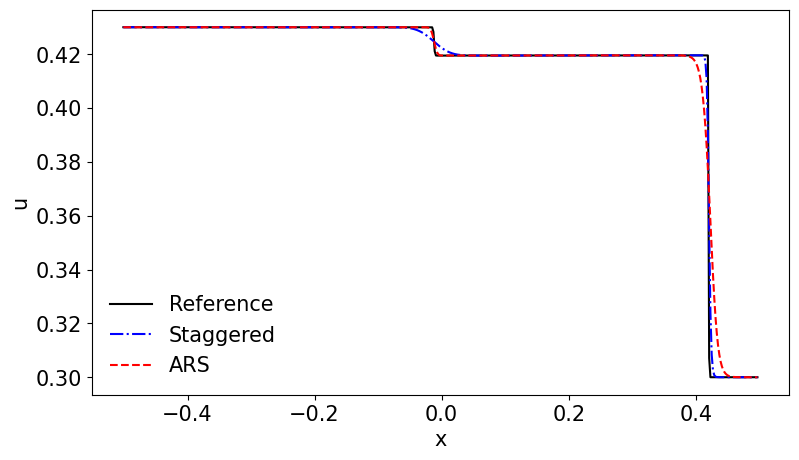}
     \includegraphics[width=0.7\textwidth]{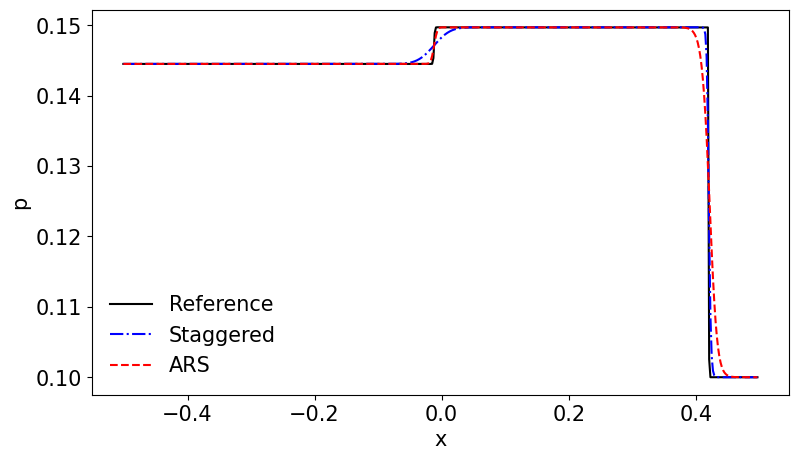}
     \includegraphics[width=0.7\textwidth]{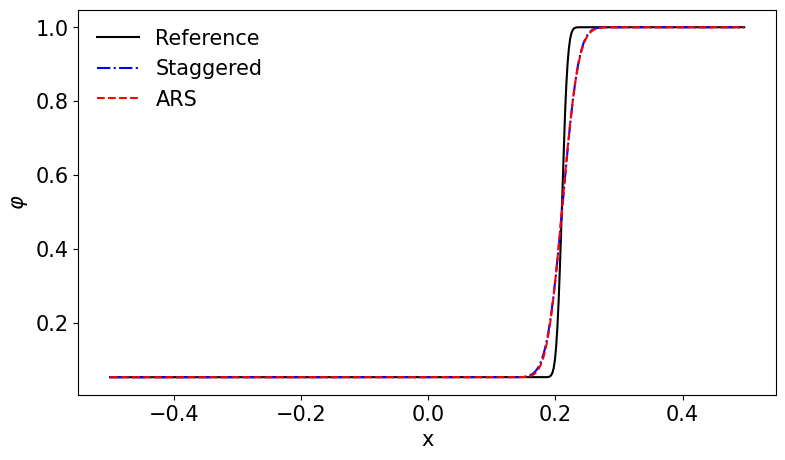}
    \caption{Profiles of density $\rho$, velocity $u$, pressure $p$ and relaxed variable $\varphi$ for the HRM model with $\varepsilon = 10^{4}$.}
    \label{fig:HRM_eps1e4}
\end{figure}

\section{Conclusion}
\label{sec:conclusion}
Two finite volume schemes have been designed for hyperbolic system of
relaxation. The main idea is to design the apporximation considering
the system as a whole, without separating the resolution of the
convective part from that of the source term.
The two schemes are asymptotic preserving in the sense that they are
consistent whatever the relaxation parameter is. In the case of the
Jin an Xin model, the preservation of invariant domains and discrete
entropy inequality are proven.
The numerical experiments illustrate the uniform performance of the schemes across stiff,
intermediate, and non-stiff regimes.

\medskip
\textbf{Acknowledgements.} This work has received the financial
support from the CNRS grant \textit{D\'efi
  Math\'ematiques France 2030}.
\bibliographystyle{plain}
\bibliography{biblio}

\begin{thebibliography}{10}

\bibitem{bereux1996zero}
F.~B{\'e}reux.
\newblock Zero-relaxation limit versus operator splitting for two-phase fluid
  flow computations.
\newblock {\em Computer Methods in Applied Mechanics and Engineering},
  133(1-2):93--124, 1996.

\bibitem{bereux1997roe}
F.~Bereux and L.~Sainsaulieu.
\newblock A roe-type riemann solver for hyperbolic systems with relaxation
  based on time-dependent wave decomposition.
\newblock {\em Numerische Math.}, 77:143--185, 1997.

\bibitem{BerthonChalons16}
C.~Berthon and C.~Chalons.
\newblock A fully well-balanced, positive and entropy-satisfying {G}odunov-type
  method for the shallow-water equations.
\newblock {\em Math. Comp.}, 85(299):1281--1307, 2016.

\bibitem{berthon2016well}
C.~Berthon, A.~Crestetto, and F.~Foucher.
\newblock A well-balanced finite volume scheme for a mixed hyperbolic/parabolic
  system to model chemotaxis.
\newblock {\em J. of Sci. Comp.}, 67:618--643, 2016.

\bibitem{boscarino2017imex}
S.~Boscarino, L.~Pareschi, and G.~Russo.
\newblock Implicit--explicit runge--kutta schemes for hyperbolic systems and
  kinetic equations in the diffusion limit.
\newblock {\em SIAM Journal on Scientific Computing}, 35(1):A22--A51, 2013.

\bibitem{boscarino2024ap}
S.~Boscarino and G.~Russo.
\newblock Asymptotic preserving methods for quasilinear hyperbolic systems with
  stiff relaxation: A review.
\newblock {\em SeMA Journal}, 81:3--49, 2024.

\bibitem{Bouchut05}
F.~Bouchut.
\newblock Stability of relaxation models for conservation laws.
\newblock In {\em European {C}ongress of {M}athematics}, pages 95--101. Eur.
  Math. Soc., Z\"urich, 2005.

\bibitem{chen94:relax}
G.~Q. Chen, C.~D. Levermore, and T.~P. Liu.
\newblock Hyperbolic conservation laws with stiff relaxation terms and entropy.
\newblock {\em Comm. Pure Appl. Math.}, 47(6):787--830, 1994.

\bibitem{ChenToro03}
G.-Q. Chen and E.~F. Toro.
\newblock Centered difference schemes for nonlinear hyperbolic equations.
\newblock {\em Journal of Hyperbolic Differential Equations}, 01(03):531--566,
  2004.

\bibitem{chen2003centred}
Gui-Qiang Chen and Eleuterio~F. Toro.
\newblock Centred schemes for nonlinear hyperbolic equations.
\newblock {\em Journal of Hyperbolic Differential Equations}, 1(4):531--566,
  2003.

\bibitem{dafermosBook}
C.~M. Dafermos.
\newblock {\em Hyperbolic conservation laws in continuum physics}, volume 325
  of {\em Grundlehren der Mathematischen Wissenschaften [Fundamental Principles
  of Mathematical Sciences]}.
\newblock Springer-Verlag, Berlin, third edition, 2010.

\bibitem{FM19}
G.~Faccanoni and H.~Mathis.
\newblock Admissible equations of state for immiscible and miscible mixtures.
\newblock In {\em Workshop on {C}ompressible {M}ultiphase {F}lows: derivation,
  closure laws, thermodynamics}, volume~66 of {\em ESAIM Proc. Surveys}, pages
  1--21. EDP Sci., Les Ulis, 2019.

\bibitem{FR13}
F.~Filbet and A.~Rambaud.
\newblock Analysis of an asymptotic preserving scheme for relaxation systems.
\newblock {\em ESAIM Math. Model. Numer. Anal.}, 47(2):609--633, 2013.

\bibitem{HN03}
B.~Hanouzet and R.~Natalini.
\newblock Global existence of smooth solutions for partially dissipative
  hyperbolic systems with a convex entropy.
\newblock {\em Arch. Ration. Mech. Anal.}, 169(2):89--117, 2003.

\bibitem{harten1983upstream}
A.~Harten, P.~D. Lax, and B.~van Leer.
\newblock On upstream differencing and godunov-type schemes for hyperbolic
  conservation laws.
\newblock {\em SIAM review}, 25(1):35--61, 1983.

\bibitem{hel-seg-06}
P.~Helluy and N.~Seguin.
\newblock Relaxation models of phase transition flows.
\newblock {\em M2AN Math. Model. Numer. Anal.}, 40(2):331--352, 2006.

\bibitem{hu2023uniform}
Y.~Hu and C.-W. Shu.
\newblock Uniform accuracy of imex runge--kutta schemes for linear hyperbolic
  systems with stiff relaxation.
\newblock {\em arXiv preprint}, 2023.

\bibitem{JinAP}
S.~Jin.
\newblock Efficient asymptotic-preserving ({AP}) schemes for some multiscale
  kinetic equations.
\newblock {\em SIAM J. Sci. Comput.}, 21(2):441--454 (electronic), 1999.

\bibitem{jin2010asymptotic}
S.~Jin.
\newblock Asymptotic preserving ({AP}) schemes for multiscale kinetic and
  hyperbolic equations: a review.
\newblock {\em Lecture notes for summer school on methods and models of kinetic
  theory (M\&MKT), Porto Ercole (Grosseto, Italy)}, pages 177--216, 2010.

\bibitem{jin1995relaxation}
S.~Jin and Z.~Xin.
\newblock The relaxation schemes for systems of conservation laws in arbitrary
  space dimensions.
\newblock {\em Comm. on pure and applied math.}, 48(3):235--276, 1995.

\bibitem{LS01}
C.~Lattanzio and D.~Serre.
\newblock Convergence of a relaxation scheme for hyperbolic systems of
  conservation laws.
\newblock {\em Numer. Math.}, 88(1):121--134, 2001.

\bibitem{Liu87}
T.-P. Liu.
\newblock Hyperbolic conservation laws with relaxation.
\newblock {\em Comm. Math. Phys.}, 108(1):153--175, 1987.

\bibitem{ma2024uniform}
Y.~Ma and C.~Huang.
\newblock Uniform accuracy of imex runge--kutta methods for linear hyperbolic
  relaxation systems.
\newblock {\em arXiv preprint}, 2024.

\bibitem{Martaud25}
L.~Martaud and C.~Berthon.
\newblock How to enforce an entropy inequality of (fully) well-balanced
  {G}odunov-type schemes for the shallow water equations.
\newblock {\em ESAIM Math. Model. Numer. Anal.}, 59(2):955--997, 2025.

\bibitem{Serre2000}
D.~Serre.
\newblock Relaxations semi-lin\'{e}aire et cin\'{e}tique des syst\`emes de lois
  de conservation.
\newblock {\em Ann. Inst. H. Poincar\'{e} C Anal. Non Lin\'{e}aire},
  17(2):169--192, 2000.

\bibitem{serre2009multidimensional}
D.~Serre.
\newblock Multidimensional shock interaction for a chaplygin gas.
\newblock {\em Archive for Rational Mechanics and Analysis}, 191:539--577,
  2009.

\bibitem{suliciu1998thermodynamics}
I.~Suliciu.
\newblock On the thermodynamics of rate-type fluids and phase transitions. i.
  rate-type fluids.
\newblock {\em International journal of engineering science}, 36(9):921--947,
  1998.

\bibitem{toro2009riemann}
E.~F. Toro.
\newblock {\em Riemann Solvers and Numerical Methods for Fluid Dynamics}.
\newblock Springer, 3rd edition, 2009.

\bibitem{toro1992force}
E.~F. Toro and M.~E. Vázquez-Cendón.
\newblock A simple approach for high-resolution centered schemes for hyperbolic
  conservation laws.
\newblock {\em Numerical Methods for Partial Differential Equations},
  8(4):441--457, 1992.

\bibitem{TORO2020}
E.F. Toro, B.~Saggiorato, S.~Tokareva, and A.~Hidalgo.
\newblock Low-dissipation centred schemes for hyperbolic equations in
  conservative and non-conservative form.
\newblock {\em Journal of Computational Physics}, 416:109545, 2020.

\bibitem{tzavaras}
A.~E. Tzavaras.
\newblock Relative entropy in hyperbolic relaxation.
\newblock {\em Commun. Math. Sci.}, 3(2):119--132, 2005.

\bibitem{Yongconvergence}
W.-A. Yong.
\newblock Singular perturbations of first-order hyperbolic systems with stiff
  source terms.
\newblock {\em J. Differential Equations}, 155(1):89--132, 1999.

\end{thebibliography}

\end{document}